\documentclass[11pt,reqno]{amsart}
\usepackage[a4paper, hmargin={2.7cm,2.7cm},vmargin={3.3cm,3.3cm}]{geometry}
\usepackage[english]{babel}
\usepackage{color}
\usepackage{amsmath}
\usepackage{amssymb}
\usepackage{amsfonts}
\usepackage{amsthm}
\usepackage{mathrsfs}
\usepackage{amsthm}
\usepackage[noadjust]{cite}
\usepackage{dsfont}
\usepackage{etoolbox}
\usepackage[unicode=true]{hyperref}
\usepackage[noabbrev]{cleveref}
\usepackage{esint}
\usepackage{todonotes}
\usepackage{crossreftools}
\usepackage{csquotes}

\MakeOuterQuote{"}

\usepackage{xparse}
\let\realItem\item
\makeatletter
\NewDocumentCommand\myItem{ o }{%
   \IfNoValueTF{#1}%
      {\realItem}
      {\realItem[#1]\def\@currentlabel{#1}}
}
\makeatother

\usepackage{enumitem}    
\setlist[enumerate]{
  topsep=1pt,
  leftmargin=30pt,
  before=\let\item\myItem,
  label=\textnormal{(\arabic*)},
  widest=(a2')
}

\binoppenalty=\maxdimen
\relpenalty=\maxdimen
\makeatletter
\patchcmd{\ttlh@hang}{\parindent\z@}{\parindent\z@\leavevmode}{}{}
\patchcmd{\ttlh@hang}{\noindent}{}{}{}
\makeatother

\newcommand\numberthis{\addtocounter{equation}{1}\tag{\theequation}}

\makeatletter
\@namedef{subjclassname@2020}{\textup{2020} Mathematics Subject Classification}
\makeatother

\theoremstyle{plain}
\newtheorem{theorem}{Theorem}[section]
\newtheorem{lemma}[theorem]{Lemma}
\newtheorem{proposition}[theorem]{Proposition}
\newtheorem{corollary}[theorem]{Corollary}

\newtheorem{innercustomgeneric}{\customgenericname}
\providecommand{\customgenericname}{}
\newcommand{\newcustomtheorem}[2]{%
  \newenvironment{#1}[1]
  {%
   \renewcommand\customgenericname{#2}%
   \renewcommand\theinnercustomgeneric{##1}%
   \innercustomgeneric
  }
  {\endinnercustomgeneric}
}

\def\XXint#1#2#3{{\setbox0=\hbox{$#1{#2#3}{\int}$ }
\vcenter{\hbox{$#2#3$ }}\kern-.6\wd0}}

\newcustomtheorem{customthm}{Theorem}
\newcustomtheorem{customlemma}{Lemma}

\theoremstyle{definition}

\theoremstyle{remark}
\newtheorem{remark}[theorem]{Remark}
\newtheorem*{remark*}{Remark}

\numberwithin{equation}{section}

\DeclareMathOperator*{\loc}{loc}

\DeclareMathOperator*{\supp}{supp}

\newcommand{\LHAs}{\mathit{h}^p_s(A)}
\newcommand{\LHA}{\mathit{h}^p(A)}
\newcommand{\LHB}{\mathit{h}^p(B)}
\newcommand{\LHBo}{\mathit{h}^1(B)}
\newcommand{\LHAo}{\mathit{h}^1(A)}

\usepackage{xparse}

\newcommand{\suppc}{\overline{\supp} \;}
\newcommand{\Schwartz}{\mathcal{S}}

\newcommand{\Mloc}{M^{0,\loc}_{\phi,A}}
\newcommand{\Indicator}{\mathds{1}}

\newcommand{\TLzero}{\mathbf{F}^{0}_{p,q}}
\newcommand{\TLone}{\mathbf{F}^{\alpha}_{p_1,q_1}}
\newcommand{\TLtwo}{\mathbf{F}^{\beta}_{p_2,q_2}}

\newcommand{\TL}{\mathbf{F}^{\alpha}_{p,q}}
\NewDocumentCommand\TLseq{O{\alpha}}{\mathbf{f}^{#1}_{p,q}}
\newcommand{\TLA}{\mathbf{F}^{\alpha}_{p,q}(A)}
\newcommand{\TLB}{\mathbf{F}^{\alpha}_{p,q}(B)}

\newcommand{\TLseqA}{\mathbf{f}^{\alpha}_{p,q}(A)}
\NewDocumentCommand\PTseq{O{\alpha}D<>{\beta}}{\mathbf{p}^{#1, #2}_{\infty,q}}

\newcommand{\vertiii}[1]{{\left\vert \kern-0.25ex
                            \left\vert \kern-0.25ex
                              \left\vert #1\right\vert\kern-0.25ex
                            \right\vert \kern-0.25ex
                          \right\vert}}

\NewDocumentCommand\DoubleStar{O{\varphi}m}{#1_{#2,\eta}^{\ast\ast}}

\newcommand{\eps}{\varepsilon}
\renewcommand{\emptyset}{\varnothing}

\newcommand{\CalB}{\mathcal{B}}

\newcommand{\Measure}{\mathrm{m}}
\newcommand{\Lebesgue}[1]{\Measure( #1)}
\newcommand{\GL}{\operatorname{GL}}
\newcommand{\Fourier}{\mathcal{F}}

\DeclareFontFamily{U}{mathx}{\hyphenchar\font45}
\DeclareFontShape{U}{mathx}{m}{n}{
      <5> <6> <7> <8> <9> <10>
      <10.95> <12> <14.4> <17.28> <20.74> <24.88>
      mathx10
      }{}
\DeclareSymbolFont{mathx}{U}{mathx}{m}{n}
\DeclareFontSubstitution{U}{mathx}{m}{n}
\DeclareMathAccent{\widecheck}{0}{mathx}{"71}
\DeclareMathAccent{\wideparen}{0}{mathx}{"75}

\newcommand{\R}{\mathbb{R}}

\newcommand{\CC}{\mathbb{C}}
\newcommand{\N}{\mathbb{N}}
\newcommand{\Z}{\mathbb{Z}}
\newcommand{\EE}{\mathbb{E}}

\title[Classification of anisotropic local Hardy spaces and Triebel-Lizorkin spaces]{Classification of anisotropic local Hardy spaces \\ and inhomogeneous Triebel-Lizorkin spaces}

\author{Jordy Timo van Velthoven}

\address{Faculty of Mathematics,
University of Vienna,
Oskar-Morgenstern-Platz 1,
A-1090 Vienna, Austria}

\email{jordy.timo.van.velthoven@univie.ac.at}

\author{Felix Voigtlaender}

\address[FV]{Mathematical Institute for Machine Learning and Data Science (MIDS),
Catholic University of Eichstätt–Ingolstadt (KU),
Auf der Schanz 49, 85049 Ingolstadt, Germany}

\email{felix@voigtlaender.xyz; felix.voigtlaender@ku.de}



\begin{document}

\begin{abstract}
This paper provides a characterization of when two expansive matrices
yield the same anisotropic local Hardy and inhomogeneous Triebel-Lizorkin spaces.
The characterization is in terms of the coarse equivalence of certain quasi-norms
associated to the matrices.
For nondiagonal matrices, these conditions are strictly weaker
than those classifying the coincidence of the corresponding homogeneous function spaces.
The obtained results complete the classification of anisotropic Besov and Triebel-Lizorkin spaces
associated to general expansive matrices.
\end{abstract}

\maketitle

\section{Introduction}

For an expansive matrix $A \in \mathrm{GL}(d, \R)$, consider Schwartz functions
$\varphi, \Phi \in \Schwartz(\R^d)$ whose Fourier transforms
$\widehat{\varphi}, \widehat{\Phi}$ satisfy the support conditions
\[
  \suppc \widehat{\varphi} \subseteq (- \tfrac{1}{2}, \tfrac{1}{2})^d \setminus \{0\}
  \quad \text{and} \quad
  \suppc \widehat{\Phi} \subseteq (-\tfrac{1}{2}, \tfrac{1}{2})^d,
\]
and the positivity condition
\[
  \sup_{i \in \N} \max \{ |\widehat{\varphi} ((A^*)^{-i} \xi)|, \; |\widehat{\Phi} (\xi)| \} > 0
  \quad \text{for all} \quad \xi \in \R^d.
\]
The associated \emph{inhomogeneous Triebel-Lizorkin space} $\TL(A)$ with $\alpha \in \R$,
$p \in (0,\infty)$ and $q \in (0, \infty]$ is defined as the space
of all tempered distributions $f \in \mathcal{S}'(\R^d)$ such that
\begin{equation} \label{eq:TL_def_intro}
  \| f \ast \Phi \|_{L^p}
  + \bigg\|
      \bigg(
        \sum_{i=1}^{\infty}
          (|\det A|^{\alpha i}
          |f \ast \varphi^A_i |)^q
      \bigg)^{1/q}
    \bigg\|_{L^p}
  < \infty,
\end{equation}
where $\varphi_i^A := |\det A|^i \varphi(A^i \cdot)$ for $i \in \N$,
with the usual modification for $q = \infty$.
For a general expansive matrix $A$, the spaces $\TL(A)$ were first introduced in \cite{bownik2006atomic}
and have been further studied in, e.g., \cite{betancor2010anisotropic, bownik2007anisotropic, liu2019littlewood1, hu2013littlewood}.
The scale of Triebel-Lizorkin spaces $\TL(A)$ includes, among others,
the Lebesgue spaces $L^p =  \mathbf{F}^0_{p, 2}(A)$ for $p \in (1,\infty)$,
and the anisotropic local Hardy spaces $\LHA = \mathbf{F}^0_{p, 2} (A)$ for $p \leq 1$;
see Section \ref{sec:hardy} for its definition.

The aim of the present paper is to determine when two expansive matrices $A, B \in \mathrm{GL}(d, \R)$
define the same inhomogeneous Triebel-Lizorkin space $\TL(A) = \TL(B)$.
For diagonal matrices with positive anisotropy, the question
of whether the associated Triebel-Lizorkin space depends on the choice of such anisotropy was considered
in \cite{triebel2004wavelet} (see also \cite[Section 5.3]{triebel2006theory}).
For two such matrices $A$ and $B$, it can be shown that the associated spaces $\TL(A)$ and $\TL(B)$
coincide precisely if $A = B^{c}$ for some $c > 0$;
or if $p \in (1,\infty)$, $q = 2$, and $\alpha = 0$, so that $L^p =  \mathbf{F}^0_{p, 2}(A) = \mathbf{F}^0_{p, 2}(B)$.
The same question for function spaces associated to general expansive matrices
is more delicate and was investigated first for anisotropic Hardy spaces $H^p(A)$,
$p \in (0,1]$ (see Section \ref{sec:hardy} for a definition):
In \cite[Chapter 1, Theorem 10.5]{bownik2003anisotropic},
it was shown that $H^p(A) = H^p(B)$ for some (equivalently, all)
$p \in (0,1]$ if, and only if, two homogeneous quasi-norms $\rho_A$ and $\rho_B$
associated to $A$ and $B$ are equivalent, in the usual sense of quasi-norms.
Corresponding results for homogeneous anisotropic Besov and Triebel-Lizorkin spaces
were only more recently obtained in \cite{cheshmavar2020classification}
and \cite{koppensteiner2023classification}, respectively.

In contrast to the case of \emph{homogeneous} function spaces, the equivalence
of two homogeneous quasi-norms $\rho_A$ and $\rho_B$ corresponding to
general expansive matrices $A$ and $B$ turns out to be \emph{not} necessary in general
for the coincidence of the associated \emph{inhomogeneous} function spaces.
More precisely, in \cite[Theorem 6.4]{cheshmavar2020classification},
it is shown that two inhomogeneous anisotropic Besov spaces defined by $A$ and $B$ coincide
if and only if the quasi-norms $\rho_{A^*}$ and $\rho_{B^*}$ associated to
the adjoints $A^*$ and $B^*$ are \emph{coarsely} equivalent,
which can be understood as the quasi-norms being merely equivalent at infinity (see  \Cref{sec:QuasiNorms}).
For simplicity, two expansive matrices $A$ and $B$ are said to be \emph{(coarsely) equivalent}
if their associated quasi-norms $\rho_A$ and $\rho_B$ are (coarsely) equivalent.
We mention that  various explicit and verifiable criteria for the (coarse) equivalence
of two matrices $A$ and $B$ in terms of spectral properties are contained in
\cite[Chapter 1, Section 10]{bownik2003anisotropic} and \cite[Section 7]{cheshmavar2020classification}.

In the present paper, we provide a refinement of the approach towards the classification
of homogeneous spaces \cite{koppensteiner2023classification}, and show that matrices
yielding the same scale of \emph{inhomogeneous} Triebel-Lizorkin spaces are characterized by \emph{coarse} equivalence.
Our main result is the following theorem, proven in \Cref{sub:MainTheoremProof}:

\begin{theorem} \label{thm:main}
  Let $A, B \in \GL(d, \R)$ be expansive.
  The following assertions are equivalent:
  \begin{enumerate}[label=\textnormal{(\roman*)},leftmargin=24pt]
    \item $\TL(A) = \TL(B)$ for some $(\alpha,p,q) \!\in\! \R \times (0,\infty) \times (0,\infty]$
          with $(\alpha, p, q) \!\notin\! \{ 0 \} \times (1,\infty) \times \{ 2 \}$;

    \item $\TL(A) = \TL(B)$ for all $\alpha \in \R$, $p \in (0, \infty)$, and $q \in (0,\infty]$;

    \item  $A^*$ and $B^*$ are coarsely equivalent.
  \end{enumerate}
\end{theorem}

\Cref{thm:main} complements the classification of homogeneous Triebel-Lizorkin spaces in \cite{koppensteiner2023classification},
and the classification of homogeneous and inhomogeneous Besov spaces in \cite{cheshmavar2020classification}.
Combined with these previous results, \Cref{thm:main} completes the classification
of all anisotropic Besov and Triebel-Lizorkin spaces introduced in \cite{bownik2005atomic, bownik2006atomic}.
In the particular case $\alpha = 0$, $p \in (0, 1]$ and $q = 2$, \Cref{thm:main}
provides also a new result for anisotropic \emph{local} Hardy spaces \cite{betancor2010anisotropic},
and complements the classification of (nonlocal) anisotropic Hardy spaces in \cite{bownik2003anisotropic}.

The proof method for establishing \Cref{thm:main} follows the overall structure
of the classification of homogeneous Triebel-Lizorkin spaces in \cite{koppensteiner2023classification}.
The key ingredients for the sufficient condition on matrices are maximal inequalities
involving a Peetre-type maximal function (cf. \Cref{sec:sufficient}),
and the necessary condition proceeds by establishing norm estimates for auxiliary functions
and reduction to the case $p = 2$ using Khintchine's inequality (cf.\ \Cref{sec:necessary}).
Our arguments for the case $\alpha = 0$, $p \in (0,1]$ and $q = 2$ follow the overall proof structure
of \cite[Chapter 1, Theorem 10.5]{bownik2003anisotropic}, while adding a significant detail for the case $p=1$ that was missing in \cite{bownik2003anisotropic} (see \Cref{rem:example}).

Despite the similarities in the overall proof structure, the arguments in the inhomogeneous case
are  more subtle and need to be more refined than their counterparts for homogeneous function spaces
in \cite{koppensteiner2023classification,bownik2003anisotropic}, for at least the following two reasons:
\begin{enumerate}[label=\textnormal{(\arabic*)}]
 \item[1)] The coarse equivalence of $A$ and $B$ does not imply their equivalence;

 \item[2)] The coarse equivalence of $A$ and $B$ is not equivalent to that of $A^*$ and $B^*$.
\end{enumerate}
The equivalence of quasi-norms and the stability of equivalence under taking adjoints
are properties repeatedly used in \cite{koppensteiner2023classification}.
Although the notions of equivalence and coarse equivalence are equivalent for diagonal matrices,
this is not necessarily the case for nondiagonal matrices (see \cite[Remark 7.10]{cheshmavar2020classification}).
As such, various parts of the arguments in \cite{bownik2003anisotropic, koppensteiner2023classification}
require nontrivial changes and new ideas in the inhomogeneous case, which we point out throughout the text.

Lastly, we mention that as in the homogeneous case \cite{koppensteiner2023classification},
it appears that the classification of inhomogeneous Triebel-Lizorkin spaces cannot be  deduced
from the general framework of Besov-type decomposition spaces \cite{voigtlaender2023embeddings},
unlike the case of anisotropic Besov spaces \cite{cheshmavar2020classification}.

The organization of the paper is as follows:
\Cref{sec:expansive} and \Cref{sec:anisotropicfunction} are devoted to background material
on expansive matrices and inhomogeneous function spaces, respectively.
The sufficient condition for the classification of matrices is proven in \Cref{sec:sufficient},
and the necessary condition is proven in \Cref{sec:necessary}.
Some technical results are postponed to the appendices.

\subsection*{Notation}

For two functions $f_1, f_2 : X \to [0, \infty)$ on a set $X$,
we write $f_1 \lesssim f_2$ whenever there exists a constant $C > 0$
such that $f_1 (x) \leq C f_2 (x)$ for all $x \in X$.
The notation $f_1 \asymp f_2$ is used to denote that $f_1 \lesssim f_2$ and $f_2 \lesssim f_1$.
For a function $f : X \to \mathbb{C}$, we denote its (possibly nonclosed)
support by $\supp f := \{ x \in X : f(x) \neq 0 \}$ and denote its closure by $\suppc f$.

The Euclidean norm of a vector $x \in \R^d$ is denoted by $|x|$,
and we write $\mathcal{B}(x, r)$ for the associated open Euclidean ball of radius $r > 0$
and center $x \in \R^d$.
The Lebesgue measure of a measurable set $\Omega \subseteq \R^d$ is denoted by $\Measure(\Omega)$.
We write $\N := \{ k \in \Z : k \geq 1 \}$ and $\N_0 := \N \cup \{ 0 \}$.
For a multi-index $\sigma \in \N_0^d$, we define its length by $|\sigma| := \sum_{j=1}^d \sigma_j$.

The Fourier transform of a function $f \in L^1 (\R^d)$ is defined as
$\widehat{f}(\xi) = \int_{\R^d} f(x) e^{-2\pi i x \cdot \xi} \; dx$ for $\xi \in \R^d$,
where $x \cdot \xi$ denotes the ordinary dot product.
We also use the notation $\mathcal{F}$ and $\mathcal{F}^{-1}$ for the Fourier transform and its inverse.
Recall that the Fourier transform restricts to a continuous linear map
$\mathcal{F} : \Schwartz(\R^d) \to \Schwartz(\R^d)$
on the space $\Schwartz(\R^d)$ of Schwartz functions, and by duality to a continuous linear map
$\mathcal{F} : \Schwartz' (\R^d) \to \Schwartz' (\R^d)$ on the space $\Schwartz' (\R^d)$
of tempered distributions, given by $\widehat{\phi} (f) := \phi(\widehat{f}\,)$
for $\phi \in \Schwartz' (\R^d)$ and $f \in \Schwartz (\R^d)$.

For $f : \R^d \to \CC$, we define $f^\ast : \R^d \to \CC$ by
$f^\ast (x) = \overline{f(-x)}$.
The translation and modulation of a function $f : \R^d \to \mathbb{C}$
are defined as $T_y f(x) = f (x-y)$ and $M_{\xi} f(x) = e^{2\pi i x \cdot \xi} f(x)$
for $x, y, \xi \in \R^d$.
For $p \in (0,\infty)$ and a matrix $M \in \mathrm{GL}(d, \R)$,
we define the associated dilation by $D^p_M f(x) = |\det M|^{1/p} f(Mx)$.
For $A \in \R^{d \times d}$, we write $A^\ast := A^T$ for the transpose of $A$.

\section{Expansive matrices, homogeneous quasi-norms and inhomogeneous covers}
\label{sec:expansive}

This section provides background on expansive matrices and their associated spaces of homogeneous type.
In addition, various properties of covers generated by powers of expansive matrices are provided.
References for the material in this section are, e.g.,
\cite{bownik2003anisotropic, cheshmavar2020classification}.

\subsection{Expansive dilations}
\label{sec:ExpansiveMatrices}

Given a  matrix $A \in \mathbb{R}^{d \times d}$, its spectrum is denoted by $\sigma(A) \subseteq \mathbb{C}$.
A matrix $A \in \mathrm{GL}(d, \R)$ is said to be \emph{expansive}
if $|\lambda| > 1$ for all $\lambda \in \sigma(A)$.

Throughout, for an expansive matrix $A$,
let $\lambda_-(A)$ and $\lambda_+(A)$ denote two fixed numbers satisfying
\[
  1 < \lambda_-(A) < \min_{\lambda \in \sigma(A)} |\lambda|
  \qquad \text{and} \qquad
  \lambda_+(A) > \max_{\lambda \in \sigma(A)} |\lambda| ,
\]
and let $\zeta_+(A) := \ln \lambda_+(A) / \ln |\det A|$ and $ \zeta_-(A) := \ln \lambda_-(A) / \ln |\det A|$.

If $A$ is an expansive matrix, then there exists an ellipsoid $\Omega_A$, that is,
a set of the form $\Omega_{A} = \{ x \in \mathbb{R}^d : |P x| < 1 \}$
for some $P \in \mathrm{GL}(d, \mathbb{R})$, and there exists some $r > 1$ such that
\begin{align} \label{eq:expansive_ellipsoid}
  \Omega_A \subseteq r \Omega_A \subseteq A \Omega_A,
\end{align}
and, additionally, $\Lebesgue{\Omega_A} = 1$, cf.\ \cite[Chapter 1, Lemma 2.2]{bownik2003anisotropic}.
The choice of an ellipsoid satisfying \eqref{eq:expansive_ellipsoid} is not necessarily unique.
For this reason, given an expansive matrix $A$, we will fix one choice of ellipsoid $\Omega_A$
associated to $A$.

\subsection{Homogeneous quasi-norms}
\label{sec:QuasiNorms}

A \emph{homogeneous quasi-norm} associated to an expansive matrix $A$ is a measurable function
$\rho_A : \mathbb{R}^d \to [0,\infty)$ satisfying:
\begin{enumerate}
  \item[(q1)] \label{enu:QuasiNormDefinite}
              $\rho_A (x) = 0$ if and only if $x = 0$;

  \item[(q2)] \label{enu:QuasiNormDilation}
              $\rho_A (A x) = |\det A| \rho_A(x)$ for all $x \in \mathbb{R}^d$;

  \item[(q3)] \label{enu:QuasiNormTriangle}
              there exists $C > 0$ such that $\rho_A(x+y) \leq C  (\rho_A(x) + \rho_A(y))$
              for all $x, y \in \mathbb{R}^d$.
\end{enumerate}
Two homogeneous quasi-norms $\rho_A, \rho_B$ associated to expansive matrices $A$ and $B$
are said to be \emph{equivalent} if there exists $C > 0$ such that
\begin{align} \label{eq:equiv_homnorms}
  \frac{1}{C} \rho_A (x) \leq \rho_B (x) \leq C \rho_A (x) \quad \text{for all} \quad x \in \R^d.
\end{align}
Similarly, two homogeneous quasi-norms $\rho_A$ and $\rho_B$ associated to $A$ and $B$
are said to be \emph{coarsely equivalent} if there exist constants $C > 0$ and $R \geq 0$ such that
\begin{align} \label{eq:coarse_equivalent}
 \frac{1}{C} \rho_A (x)  - R \leq \rho_B (x) \leq C \rho_A (x) + R
 \qquad \text{for all } x \in \R^d.
\end{align}
Clearly, any two equivalent quasi-norms are also coarsely equivalent,
but the converse is not true in general, cf.\ \cite[Remark 7.10]{cheshmavar2020classification}. 

By \cite[Chapter 1, Lemma 2.4]{bownik2003anisotropic}, any two quasi-norms $\rho_A, \rho'_A$
associated to a fixed matrix $A$ are equivalent.
We will simply say that two expansive matrices $A$ and $B$ are \emph{equivalent}
(resp.\ \emph{coarsely equivalent}) if their associated quasi-norms are equivalent
(resp.\ coarsely equivalent).

In the sequel, we work with a specific choice of quasi-norm;
namely, we will use the so-called \emph{step homogeneous quasi-norm}
$\rho_A$ associated to $A$, defined by
\[
  \rho_A (x)
  = \begin{cases}
      |\det A|^i, & \text{if} \quad x \in A^{i+1} \Omega_A \setminus A^i \Omega_A, \\
      0,          & \text{if} \quad x = 0,
    \end{cases}
\]
where $\Omega_A$ is the fixed ellipsoid from \eqref{eq:expansive_ellipsoid};
see \cite[Chapter 1, Definition~2.5]{bownik2003anisotropic}.
For this quasi-norm, it is easy to see that it is symmetric,
in the sense that $\rho_A (x) = \rho_A(-x)$ for all $x \in \R^d$.

Lastly, we state the following characterization of coarse equivalence of two matrices,
which we will use in the proof of the main theorem.
See \cite[Lemma 4.10]{cheshmavar2020classification} for a proof.

\begin{lemma}[\cite{cheshmavar2020classification}] \label{lem:coarse_norm}
  Let $A, B \in \mathrm{GL}(d, \R)$ be expansive.
  Then
  $A$ and $B$ are coarsely equivalent if and only if
  \[
    \sup_{k \in \mathbb{N}}
      \big\| A^{-k} B^{\lfloor \varepsilon k \rfloor} \big\|
    < \infty,
  \]
  where $\varepsilon  = \varepsilon (A, B) := \ln |\det A| / \ln |\det B|$.
\end{lemma}

\subsection{Inhomogeneous covers}
\label{sec:covers}

Let $A \in \mathrm{GL}(d, \R)$ be an expansive matrix and fix an open set $Q \subseteq \mathbb{R}^d$
with compact closure $\overline{Q} \subseteq \mathbb{R}^d \setminus \{0\}$.
An \emph{inhomogeneous cover induced by $A$} is a family $ (Q^A_i)_{i \in \N_0}$
of sets $Q^A_i \subseteq \R^d$, where $Q^A_i = A^i Q$ for  $i \geq 1$, and $Q^A_0 \subseteq \R^d$
is any relatively compact open set with the property that $\bigcup_{i \in \N_0} Q^A_i = \R^d$.

For two inhomogeneous covers $(Q^A_i)_{i \in \N_0}$ and $(P^B_j)_{j \in \N_0}$
induced by expansive matrices $A$ and $B$ respectively, define, for fixed $i, j \in \N_0$,
the index sets
\begin{align} \label{eq:index1}
   J_i := \{ \ell \in \N_0 : Q^A_i \cap P^B_\ell \neq \emptyset \}
   \quad \text{and} \quad
   I_j := \{ \ell \in \N_0 : Q^A_\ell \cap P^B_j \neq \emptyset \}.
\end{align}
Moreover, given $i \in \N_0$ and $n \in \N_0$, define the index sets $i^{n*} \subseteq \N_0$ inductively as
\begin{align} \label{eq:index2}
  i^{0 *} := \{ i \}
  \qquad \text{and} \qquad
  i^{(n+1)*} := \{ j \in \N_0 : Q^A_k \cap Q^{A}_j \neq \emptyset \;\; \text{for some} \;\; k \in i^{n*} \}.
\end{align}
We will also often simply write $i^*$ for $i^{1*}$.
If we need to make clear whether the sets $i^{n \ast}$ are computed with respect to the cover
$(Q_i^A)_{i \in \N_0}$ or the cover $(P_j^B)_{j \in \N_0}$,
we write $i^{n \ast A}$ or $i^{n \ast B}$.

Following the terminology of \cite{ voigtlaender2023embeddings},
the cover $(Q^A_i)_{i \in \N_0}$ is said to be \emph{almost subordinate to} $(P^B_j)_{j \in \N_0}$
if there exists $k \in \N_0$ such that for every $i \in \N_0$ there exists $j_i \in \N_0$
with $Q_i^A \subseteq \bigcup_{j \in j_i^{k \ast}} P^B_j$.
In addition, the covers $(Q^A_i)_{i \in \N_0}$ and $(P^B_j)_{j \in \N_0}$
are said to be \emph{equivalent} if $(Q^A_i)_{i \in \N_0}$ is almost subordinate
to $(P^B_j)_{j \in \N_0}$,  and vice versa.

The following result provides a characterization of the coarse equivalence of two matrices
in terms of geometric properties of their associated inhomogeneous covers;
cf.\ \cite[Lemma 6.3]{cheshmavar2020classification}.
These properties are the ones that will actually be used/verified in our proof of \Cref{thm:main}.

\begin{lemma}[\cite{cheshmavar2020classification}] \label{lem:cover_coarse}
  Let $A, B \in \mathrm{GL}(d, \R)$ be expansive matrices and let $(Q^A_i)_{i \in \N_0}$
  and $(P^B_j)_{j \in \N_0}$ be inhomogeneous covers induced by $A$ and $B$, respectively.
  Then the following assertions are equivalent:
  \begin{enumerate}
   \item[(i)]  $A$ and $B$ are coarsely equivalent;
   \item[(ii)] $(Q^A_i)_{i \in \N_0}$ and $(P^B_j)_{j \in \N_0}$ are equivalent;
   \item[(iii)] $\sup_{i \in \N_0} |J_i| + \sup_{j \in \N_0} |I_j| < \infty$.
  \end{enumerate}
\end{lemma}

In the remainder of this subsection, we collect several additional observations about the index sets
defined in \Cref{eq:index1} and \Cref{eq:index2} that will be used later.
We begin with the following inclusion property for the sets defined in \Cref{eq:index2}.
Its proof is similar, but not identical, to that of \cite[Lemma~5.2]{cheshmavar2020classification}.

\begin{lemma}\label{lem:NeighborSetControl}
  Let $A \in \GL(d,\R)$ be expansive and let $(Q_i^A)_{i \in \N_0}$ be an inhomogeneous
  cover induced by $A$.
  Then there exists $M \in \N$ such that, for all $i \in \N_0$,
  \[
    i^\ast \subseteq \{ \ell \in \N_0 \,\,:\,\, |\ell - i| \leq M \}.
  \]
\end{lemma}

\begin{proof}
  By definition of an inhomogeneous cover induced by $A$, there exists an open set
  $Q \subset \R^d$ with compact closure $\overline{Q} \subseteq \R^d \setminus \{ 0 \}$
  and such that $Q_j^A = A^j Q$ for all $j \in \N$.
  Moreover, $Q_0^A \subseteq \R^d$ is open and relatively compact.
  Thus, we can choose $R > 0$ sufficiently large such that
  \[
    Q_0^A \subseteq \overline{\CalB}(0, R)
    \qquad \text{and} \qquad
    Q \subset C_R := \{ x \in \R^d \,\,:\,\, \tfrac{1}{R} \leq |x| \leq R \}
    .
  \]
  By \cite[Chapter 1, Section 2]{bownik2003anisotropic}, there exists a constant $c > 0$ satisfying
  \[
    \frac{1}{c} \, \lambda_-^j \, |x|
    \leq |A^j x|
    \leq c \, \lambda_+^j \, |x|
    \qquad \text{for all } j \in \N_0 \text{ and } x \in \R^d
    ,
  \]
  where $\lambda_\pm = \lambda_\pm (A) > 1$ are as in \Cref{sec:ExpansiveMatrices}.
  Fix some $M \in \N$ so large that
  \[
    M \geq \ln (c R^2) / \ln (\lambda_-).
  \]
  The remainder of the proof is divided into two cases, which together easily imply the claim.
  \\~\\
  \emph{Case 1.} We show that if $i,\ell \in \N$ satisfy $Q_i^A \cap Q_\ell^A \neq \emptyset$,
  then $|i - \ell| \leq M$.
  By symmetry, we can clearly assume that $\ell \geq i$.
  Since $\emptyset \neq A^i Q \cap A^\ell Q$, and thus
  \[
    \emptyset \neq Q \cap A^{\ell - i} Q \subseteq C_R \cap A^{\ell - i} C_R
    ,
  \]
   there exists some $x \in C_R$ such that $A^{\ell - i} x \in C_R$ as well.
  But this implies
  \[
    R 
    \geq |A^{\ell - i} x|
    \geq \frac{1}{c} \, \lambda_-^{\ell - i} \, |x|
    \geq \frac{1}{c R} \lambda_{-}^{\ell - i} ,
  \]
  and this easily implies $0 \leq \ell - i \leq \ln (c R^2) / \ln (\lambda_-) \leq M$,
  as desired.
  \\~\\
  \emph{Case 2.}
  If $Q_0^A \cap Q_i^A \neq \emptyset$ for some $i \in \N$,
  then there exists $x \in Q \subseteq C_R$ satisfying $A^i x \in Q_0^A \subseteq \overline{\CalB} (0, R)$.
  Hence,
  \[
    R
    \geq |A^i x|
    \geq \frac{1}{c} \, \lambda_-^i \, |x|
    \geq \frac{1}{c R} \, \lambda_-^i
    ,
  \]
  which yields $i \leq \ln(c R^2) / \ln (\lambda_-) \leq M$.
\end{proof}

As a consequence of the previous two lemmata, we obtain the following corollary.

\begin{corollary} \label{cor:cover_coarse}
  With notation as in \Cref{lem:cover_coarse}, the following holds:
  If $A$ and $B$ are coarsely equivalent, there exists a constant $C > 0$
  such that whenever $Q^A_i \cap P^B_j \neq \emptyset$ for some $i, j \in \N_0$, then
  \[
   \frac{1}{C} | \det B|^j \leq |\det A|^i \leq C | \det B|^j.
  \]
\end{corollary}

\begin{proof}
For ease of notation, let us set $P_j^B := P_0$ for $j \in \Z$ with $j<0$.
If $A$ and $B$ are coarsely equivalent, then the covers $(Q^A_i)_{i \in \N_0}$ and $(P^B_j)_{j \in \N_0}$
are equivalent by \Cref{lem:cover_coarse}.
Hence, there exists $k \in \N$ such that for every $i \in \N_0$, there exists
$j_i \in \N_0$ with $Q_i^A \subset \bigcup_{\ell \in j_i^{k \ast B}} P_\ell^B$.
As an easy consequence of \Cref{lem:NeighborSetControl}, there exists $M \in \N$ such that
$j^{(2 k+1)\ast B} \subseteq \{ \ell \in \N_0 \,:\, |\ell - j| \leq M \}$
for all $j \in \N_0$.

Let $i,j \in \N_0$ be such that
$\emptyset \neq Q_i^A \cap P_j^B \subseteq \bigcup_{\ell \in j_i^{k \ast B}} (P_\ell^B \cap P_j^B)$.
Then $\emptyset \neq P_\ell^B \cap P_j^B$ for some $\ell \in j_i^{k \ast B}$,
and hence $j_i \in \ell^{k \ast B} \subseteq j^{(k+1) \ast B}$,
which implies
\[
  j_i^{k \ast B}
  \subseteq j^{(2 k + 1) \ast B}
  \subseteq \{ \ell \in \N_0 \,:\, |\ell - j| \leq M \}.
\]
Therefore,
\begin{equation}
  Q_i^A
  \subseteq \bigcup_{\ell \in j_i^{k \ast B}} P_\ell^B
  \subseteq \bigcup_{\ell = -M}^M P_{j + \ell}^B
  ,
  \label{eq:EquivalenceImpliesNiceCovering}
\end{equation}
and thus
\[
  |\det A|^i
  \lesssim \Measure(Q_i^A)
  \lesssim \sum_{\ell = -M}^M \Measure(P^B_{j+\ell})
  \lesssim \sum_{\ell = -M}^M |\det B|^{j+\ell}
  \lesssim |\det B|^j.
\]
The reverse inequality follows by exchanging the role of $A$ and $B$.
\end{proof}

Lastly, we state the following adaptation of a corresponding result for homogeneous covers.
Its proof is virtually identical to that of \cite[Lemma 2.5]{koppensteiner2023classification}, and hence omitted.

\begin{lemma} \label{lem:cardinality_uniform}
Let $A, B \in \mathrm{GL}(d, \R)$ be expansive matrices with associated induced inhomogeneous covers
$(Q^A_i)_{i \in \N_0}$ and $(P^B_j)_{j \in \N_0}$, respectively.
If there exists $C>0$ satisfying
\[
  \frac{1}{C} |\det A|^{i}
  \leq  |\det B|^{j}
  \leq C |\det A|^{i}
  \quad \text{for all $i,j \in \mathbb{N}_0$ for which $Q^A_i \cap P^B_j \neq \emptyset$},
\]
then there exists $N \in \mathbb{N}$ satisfying
\[
  J_i \subseteq \bigg\{ j \in \N_0 : | j - \lfloor \varepsilon i \rfloor | \leq N \bigg\}
  \quad \text{and} \quad
  I_j \subseteq \bigg\{ i \in \N_0 : \bigg| i - \bigg\lfloor  \frac{j}{\varepsilon}  \bigg \rfloor \bigg| \leq N \bigg\}
\]
for all $i, j \in \N_0$, where $\varepsilon := \ln |\det A| / \ln |\det B|$
is as in \Cref{lem:coarse_norm}.
\end{lemma}

\section{Anisotropic inhomogeneous function spaces}
\label{sec:anisotropicfunction}

This section provides various preliminary results on anisotropic local Hardy spaces
and inhomogeneous Triebel-Lizorkin spaces that are used in the proof of \Cref{thm:main}.
For further background and results on these spaces, see the papers \cite{betancor2010anisotropic, bownik2006atomic}.

\subsection{Inhomogeneous Triebel-Lizorkin spaces}
\label{sec:TL}
Let $A \in \mathrm{GL}(d, \R)$ be an expansive matrix.
A pair $(\varphi, \Phi)$ consisting of Schwartz functions $\varphi, \Phi \in \mathcal{S}(\mathbb{R}^d)$
is said to be an \emph{$A$-analyzing pair} if the Fourier transforms
$\widehat{\varphi}, \widehat{\Phi}$ satisfy%
\footnote{In most other papers, including \cite{bownik2006atomic}, the cube $[-\pi,\pi]^d$ is used
instead of $(-\frac{1}{2},\frac{1}{2})^d$. The reason for this seeming inconsistency is that
\cite{bownik2006atomic} uses a different normalization of the Fourier transform than we do.}
\begin{enumerate}
  \item[(c1)] \label{enu:AnalyzingPairSupportCondition}
              $\suppc \widehat{\varphi} \subseteq (-\frac{1}{2}, \frac{1}{2})^d \setminus \{0\}$
              and $\suppc \widehat{\Phi} \subseteq (-\frac{1}{2}, \frac{1}{2})^d$;
              \smallskip{}

  \item[(c2)] \label{enu:AnalyzingPairNonVanishingCondition}
              \(
                \sup_{i \in \mathbb{N}}
                  \max
                  \{
                    |\widehat{\varphi} ((A^*)^{-i} \xi) | , \;
                    |\widehat{\Phi}(\xi)|
                  \}
                > 0
              \)
              for all $\xi \in \mathbb{R}^d$.
\end{enumerate}
There always exists an $A$-analyzing pair $(\varphi, \Phi)$
that in addition to conditions \ref{enu:AnalyzingPairSupportCondition}
and \ref{enu:AnalyzingPairNonVanishingCondition} satisfies
the additional condition
\begin{enumerate}
  \item[(c3)] \label{enu:AnalyzingPairPartitionOfUnity}
              $\widehat{\Phi} (\xi) + \sum_{i \in \mathbb{N}} \widehat{\varphi}((A^*)^{-i} \xi) = 1$
              for all $\xi \in \R^d$;
\end{enumerate}
see, e.g., \cite[Section 3.3]{bownik2006atomic} and \cite[Remark 2.3]{cheshmavar2020classification}.

Following \cite{bownik2006atomic}, given an $A$-analyzing pair $(\varphi, \Phi)$,
$\alpha \in \mathbb{R}$, $0 < p < \infty$ and $0 < q \leq \infty$,
the associated \emph{inhomogeneous anisotropic Triebel-Lizorkin space}
$\TLA = \mathbf{F}^{\alpha}_{p,q}(A; \varphi, \Phi)$ is defined as the collection
of all tempered distributions $f \in \mathcal{S}' (\R^d)$ for which
\begin{equation} \label{eq:TL_def_prelims}
  \| f \|_{\TLA}
  := \| f \|_{\TL (A; \varphi, \Phi)}
  := \bigg\|
       \bigg(
         \sum_{i \in \N_0}
           (|\det A|^{\alpha i} |f \ast \varphi^A_i |)^q
       \bigg)^{1/q}
     \bigg\|_{L^p}
  <  \infty,
\end{equation}
where $\varphi^A_0 := \Phi$ and $\varphi^A_i := |\det A|^i \varphi (A^i \cdot)$ for $i \geq 1$,
and with the usual modification in \eqref{eq:TL_def_prelims} for $q = \infty$.
The quantity \eqref{eq:TL_def_prelims} is easily seen to be equivalent to the quasi-norm \eqref{eq:TL_def_intro},
a fact that will often be used without further mention.
The spaces $\TLA$ are well-defined, in the sense that they do not depend
on the choice of the $A$-analyzing pair $(\varphi, \Phi)$, cf.\ \cite[Section 3.3]{bownik2006atomic}.

In addition to the above properties, the spaces $\TLA$ are complete.
This property appears to be taken as self-evident in the literature,
but is never explicitly stated.
As this property is used repeatedly in the proof of our main result,
we provide a short proof in the appendix; see \Cref{lem:Completeness}.

\subsection{Local Hardy spaces}
\label{sec:hardy}

Let $A \in \mathrm{GL}(d, \mathbb{R})$ be an expansive matrix.
Given $\phi \in \mathcal{S}(\R^d)$ with $\int \phi \, dx \neq 0$,
the associated \emph{local radial maximal function} of $f \in \mathcal{S}'(\R^d)$ is defined as
\[
  \Mloc f(x)
  := \sup_{j \in \mathbb{N}_0}
       |\det A|^j
       \big| (f \ast (\phi \circ A^j)) (x) \big|,
  \quad x \in \mathbb{R}^d.
\]
The \emph{anisotropic local Hardy space} $\LHA$, with $p \in (0,\infty)$,
is the space of all $f \in \mathcal{S}' (\R^d)$ satisfying
\[
  \| f \|_{\LHA}
  := \big\| \Mloc f \big\|_{L^p}
  < \infty,
\]
and is complete with respect to the quasi-norm $\| \cdot \|_{\LHA}$.
The definition of $\LHA$ is independent of the choice of defining vector $\phi$.
If $p \in (1, \infty)$, then $\LHA = L^p$, and for $p = 1$ it holds that $\LHAo \subseteq L^1$.
See, e.g., \cite[Section 2]{betancor2010anisotropic} for these claims.

In a similar manner, the \emph{(nonlocal) anisotropic Hardy space} $H^p(A)$
is defined as the space of all $f \in \mathcal{S}'(\R^d)$ such that
\[
  \| f \|_{H^p (A)}
  := \| M^0_{\phi,A} f \|_{L^p} < \infty,
  \quad \text{where} \quad
  M_{\phi,A}^0 f(x)
  := \sup_{j \in \Z} |\det A|^j | f \ast (\phi \circ A^j) (x)|.
\]
Clearly, $H^p(A) \subseteq \LHA$, with $ \| f \|_{\LHA} \leq \| f \|_{H^p(A)} $ for all $f \in \mathcal{S}'(\R^d)$.
For $p \in (1,\infty)$, we have $L^p = H^p (A)$; see \cite[Chapter 1, Section 3]{bownik2003anisotropic}.

The following Littlewood-Paley characterization identifies
local Hardy spaces as special inhomogeneous Triebel-Lizorkin spaces.

\begin{proposition} \label{prop:TL_hardy}
Let $\varphi \in \mathcal{S}(\R^d)$ be a function such that
$\suppc \widehat{\varphi} \subseteq (-\frac{1}{2}, \frac{1}{2})^d \setminus \{0\}$ and
\[
  \sum_{i \in \Z} \widehat{\varphi}((A^*)^{-i} \xi) = 1,
  \quad \xi \in \R^d \setminus \{0\}.
\]
Define $\Phi \in \mathcal{S}(\R^d)$ by
$\widehat{\Phi}(\xi) = \sum^{0}_{i = - \infty} \widehat{\varphi}((A^*)^{-i} \xi)$
for $\xi \in \R^d \setminus \{0\}$ and $\widehat{\Phi}(0) = 1$.
Then, for every $p \in (0,\infty)$, the (quasi)-norm equivalence
\[
  \| f \|_{\LHA}
  \asymp \| f \ast \Phi \|_{L^p}
         + \bigg\|
              \bigg(\sum_{i =1}^{\infty}  |f \ast \varphi^A_i |^2 \bigg)^{1/2}
           \bigg\|_{L^p}
  \asymp \| f \|_{\mathbf{F}^{0}_{p,2}(A)}
\]
holds for all $f \in \mathcal{S}'(\R^d)$.
\end{proposition}

\begin{proof}
For $p \in (0,1]$, the claim corresponds to \cite[Theorem 1.2, Part (ii)]{hu2013littlewood}.
For $p \in (1,\infty)$, recall from above that $\LHA = L^p = H^p (A)$.
Let $f \in \mathcal{S}'(\R^d)$.
First, note that
\[
 \| f \|_{L^p}
 \asymp \| f \|_{H^p (A)}
 \asymp \bigg\|
          \bigg(
            \sum_{i \in \mathbb{Z}}
              \big(
                |\det A|^i  |f \ast (\varphi \circ A^i) |
              \big)^2
          \bigg)^{1/2}
        \bigg\|_{L^p}
  ,
\]
by a combination of \cite[Theorem 7.1]{bownik2007anisotropic} and \cite[Chapter 1, Theorem 3.9]{bownik2003anisotropic}.
It follows that
\begin{align*}
  \| f \|_{\mathbf{F}^{0}_{p,2}(A)}
  &\lesssim \| f \|_{L^p} \| \Phi \|_{L^1}
           + \bigg\|
               \bigg(
                 \sum_{i \in \mathbb{Z}}
                 \big(
                   |\det A|^i |f \ast (\varphi \circ A^i) |
                 \big)^2
               \bigg)^{1/2}
             \bigg\|_{L^p}
  \lesssim \| f \|_{L^p} \\
  &\asymp \| f \|_{\LHA}. \numberthis \label{eq:TL_Lp}
\end{align*}

The reverse inequality is an adaptation of a standard argument from Littlewood-Paley theory
to the anisotropic setting.
By \cite[Section 3.3]{bownik2006atomic}, there exists another $A$-analyzing pair $(\psi, \Psi)$ such that
\[
  f
  = f \ast \Phi \ast \Psi^*
    + \sum_{i \in \mathbb{N}} f \ast \varphi^A_i \ast (\psi^*)^A_i
\]
with convergence in $\mathcal{S}'(\R^d)$; this convergence follows from \cite[Lemma 2.6]{bownik2006atomic}
(see also \cite[Section 3.3]{bownik2006atomic}).
Using this identity, it follows that
\begin{align} \label{eq:TL_Lp2}
  \| f \|_{h^p (A)}
  \asymp \| f \|_{L^p}
  \leq \| f \ast \Phi \|_{L^p} \| \Psi^* \|_{L^1}
       + \bigg\| \sum_{i \in \N} f \ast \varphi^A_i \ast (\psi^*)^A_i \bigg\|_{L^p}.
\end{align}
For estimating the second summand, we use the dual characterization of $L^p$.
Let $\langle \cdot, \cdot \rangle$ denote the sesquilinear dual pairing between $\Schwartz' (\R^d)$
and $\Schwartz (\R^d)$, which is antilinear in the second component,
and let $p' \in (1, \infty)$ denote the conjugate exponent for $p$.
If $h \in L^{p'} \cap \Schwartz (\R^d)$, then an application of the monotone convergence theorem
and the Cauchy-Schwarz inequality gives
\begin{align*}
  \bigg| \big\langle  \sum_{i \in \N} f \ast \varphi^A_i \ast (\psi^*)^A_i, h \big \rangle \bigg|
 & \leq \sum_{i \in \N}
          \big| \big\langle  f \ast \varphi^A_i, h \ast \psi^A_i \big \rangle \big| \\
 & \leq \int_{\R^d}
          \bigg(\sum_{i \in \N} |f \ast \varphi^A_i (x) |^2\bigg)^{\frac{1}{2}}
          \bigg(\sum_{i \in \N} |h \ast \psi^A_i (x)|^2\bigg)^{\frac{1}{2}}
        \; dx \\
 & \leq \bigg\|
          \bigg(\sum_{i \in \N} |f \ast \varphi^A_i  |^2\bigg)^{\frac{1}{2}}
        \bigg\|_{L^p}
        \bigg\|
          \bigg(\sum_{i \in \N} |h \ast \psi^A_i  |^2\bigg)^{\frac{1}{2}}
        \bigg\|_{L^{p'}} \\
 & \lesssim \bigg\| \bigg(\sum_{i \in \N} |f \ast \varphi^A_i  |^2\bigg)^{\frac{1}{2}} \bigg\|_{L^p}
            \| h \|_{L^{p'}},
\end{align*}
where the penultimate step used Hölder's inequality and the last step used \Cref{eq:TL_Lp}
(for $\psi$ instead of $\varphi$ and $p'$ instead of $p$).
Thus, by the dual characterization of $L^p$, the tempered distribution
$\sum_{i \in \N} f \ast \varphi^A_i \ast (\psi^*)^A_i$  satisfies
\[
  \bigg\| \sum_{i \in \N} f \ast \varphi^A_i \ast (\psi^*)^A_i \bigg\|_{L^p}
  = \sup_{\substack{h \in \Schwartz(\R^d) \\ \| h \|_{L^{p'}} \leq 1}}
      \bigg|
        \bigg\langle  \sum_{i \in \N} f \ast \varphi^A_i \ast (\psi^*)^A_i, h \bigg\rangle
      \bigg|
  \lesssim  \bigg\|
              \bigg(\sum_{i \in \N} |f \ast \varphi^A_i  |^2\bigg)^{\frac{1}{2}}
            \bigg\|_{L^p}.
\]
In combination with \Cref{eq:TL_Lp,eq:TL_Lp2}, this finishes the proof.
\end{proof}

\subsection{Local atoms}
\label{sec:atomic}

Let $p \in (0,1]$ and $s \in \mathbb{N}$ be such that $s \geq \lfloor (\frac{1}{p} -1) \zeta_-(A)^{-1}\rfloor$.
A \emph{local $(p,s)$-atom associated to $A$} is a measurable function $a : \mathbb{R}^d \to \mathbb{C}$
such that there exist $x_0 \in \mathbb{R}^d$ and $j \in \mathbb{Z}$ satisfying:
\begin{enumerate}
 \item[(a1)] \label{enu:AtomSupportCondition}
             $\supp a \subseteq x_0 + A^j \Omega_A$;

 \item[(a2)] \label{enu:AtomLInftyCondition}
              $\| a \|_{L^\infty} \leq |\det A|^{-\frac{j}{p}}$;

 \item[(a3)] \label{enu:LocalAtomVanishingMoments}
             If $j < 0$, then $\int_{\mathbb{R}^d} a(x) x^{\sigma} \; dx = 0$
             for all $\sigma \in \mathbb{N}_0^d$ with $|\sigma| \leq s$.
\end{enumerate}
In addition, we call a measurable function $a$ merely a \emph{$(p,s)$-atom associated to $A$}
if it satisfies \ref{enu:AtomSupportCondition}, \ref{enu:AtomLInftyCondition} and
\begin{enumerate}
 \item[(a4)] \label{enu:AtomVanishingMoments}
             $\int_{\mathbb{R}^d} a(x) x^{\sigma} \; dx = 0$
             for all $\sigma \in \mathbb{N}_0^d$ with $|\sigma| \leq s$.
\end{enumerate}
Clearly, any $(p,s)$-atom is a local $(p,s)$-atom.

\begin{remark} \label{rem:alternative_atom}
A useful alternative definition of (local) atoms is as follows.
Let $p \in (0,1]$ and $s \in \mathbb{N}$ be such that $s \geq \lfloor (\frac{1}{p}-1) \zeta_-(A)^{-1} \rfloor$.
An \emph{alternative local $(p,s)$-atom (resp.\ \emph{alternative $(p,s)$-atom}) associated to $A$},
is a measurable function $a : \mathbb{R}^d \to \mathbb{C}$
such that there exist $x_0 \in \mathbb{R}^d$ and $j \in \mathbb{Z}$ satisfying:
\begin{enumerate}
 \item[(a1')] \label{enu:AlternativeAtomSupportCondition}
              $\supp a \subseteq x_0 + A^j \mathcal{B}(0,1)$,
 \item[(a2')] \label{enu:AlternativeAtomLInftyBound}
              $\| a \|_{L^\infty} \leq \Measure(A^j(\mathcal{B}(0,1)))^{-\frac{1}{p}}$,
 \end{enumerate}
and \ref{enu:LocalAtomVanishingMoments} (resp.\ \ref{enu:AtomVanishingMoments}).
Any alternative (local) $(p,s)$-atom is a constant multiple of a (local) $(p,s)$-atom and vice versa,
with a constant only depending on $p$, $A$; see \cite[Remark on page 72]{bownik2003anisotropic}.
\end{remark}

By \cite[Proposition 2.2]{betancor2010anisotropic}, the local Hardy space $\LHA$
is equal to the space of all tempered distributions $f$ of the form
\begin{align} \label{eq:atomic_decomposition}
 f = \sum_{n \in \mathbb{N}} c_n a_n
\end{align}
for a sequence $(a_n)_{n \in \mathbb{N}}$ of local $(p,s)$-atoms $a_n$ associated to $A$
and $(c_n)_{n \in \mathbb{N}} \in \ell^p (\mathbb{N})$.
In addition, for every $f \in \LHA$, the quantity
\[
  \| f \|_{\LHAs}
  := \inf \bigg\{ \| c \|_{\ell^p} : f = \sum_n c_n a_n \bigg\},
\]
where the infimum is taken over all atomic decompositions \eqref{eq:atomic_decomposition}
in terms of local $(p,s)$-atoms, is equivalent to $\| f \|_{\LHA}$.

\section{Sufficient conditions for classification}
\label{sec:sufficient}

This section is devoted to proving the sufficient condition of \Cref{thm:main}
for the equality of anisotropic inhomogeneous Triebel-Lizorkin spaces. We prove this result as \Cref{prop:sufficient} below.

\subsection{General notation} \label{sec:notation1}
Throughout this section, let $A, B \in \GL(d,\R)$ be expansive matrices
and let $(\varphi, \Phi)$ and $(\psi, \Psi)$ be pairs of analyzing vectors
satisfying conditions \ref{enu:AnalyzingPairSupportCondition}-\ref{enu:AnalyzingPairPartitionOfUnity}
for $A$ and $B$, respectively.
Define $Q_0 := \supp \widehat{\Phi}$ and $Q := \supp \widehat{\varphi}$,
and set $P_0 := \supp \widehat{\Psi}$ and $P := \supp \widehat{\psi}$.
Furthermore, define $Q_i^{A^*} := (A^*)^i Q$ and $P_j^{B^*} := (B^*)^j P$
for $i, j \geq 1$ and $Q_i^{A^*} := Q_0$ and $P_j^{B^*} := P_0$ for $i, j \leq 0$.
Then the conditions \ref{enu:AnalyzingPairSupportCondition} and \ref{enu:AnalyzingPairPartitionOfUnity}
guarantee that the families $(Q^{A^*}_i)_{i \in \N_0}$
and $(P^{B^*}_j)_{j \in \N_0}$ are inhomogeneous covers induced by $A^*$ and $B^*$, respectively.
As in \Cref{sec:covers}, we define
\[
  J_i := \{ \ell \in \N_0 : Q^{A^*}_i \cap P^{B^*}_\ell \neq \emptyset \}
  \quad \text{and} \quad
  I_j := \{ \ell \in \N_0 : Q^{A^*}_\ell \cap P^{B^*}_j \neq \emptyset \},
\]
for fixed $i, j \in \N_0$.
Lastly, set $\varphi^A_0 := \Phi$ and $\varphi^A_i := |\det A|^i \varphi(A^i \, \cdot)$ for $i \geq 1$,
and define $\psi_j^B$ for $j \in \N_0$ in a similar manner (using $B$ instead of $A$).
Note that $\supp \widehat{\varphi_i^A} = Q_i^{A^*}$
and $\supp \widehat{\psi_j^B} = P_j^{B^*}$ for $i,j \in \N_0$.

\subsection{Peetre-type inequality}
\label{sec:peetretype}

Throughout the remainder of this section, we assume that the adjoint matrices $A^*$ and $B^*$
are coarsely equivalent, in the sense of \Cref{sec:QuasiNorms}.

A central ingredient in establishing the sufficient condition of \Cref{thm:main}
is an anisotropic Peetre-type inequality involving the \emph{two} dilation matrices $A$ and $B$
(cf.\ \Cref{lem:peetretype}).
For stating this result,  recall that the \emph{anisotropic Hardy-Littlewood maximal operator} $M_{\rho_A} h$
applied to a measurable function $h : \R^d \to \mathbb{C}$ is defined by
\begin{equation}
  M_{\rho_A} h (x)
  := \sup_{\mathcal{B}_A \ni x}
       \frac{1}{\Lebesgue{\mathcal{B}_A}}
       \int_{\mathcal{B}_A} |h(y)| \; dy,
  \quad x \in \mathbb{R}^d,
  \label{eq:HLMaximalOperator}
\end{equation}
where the supremum is taken over all $\rho_A$-balls
$\mathcal{B}_A = \mathcal{B}_{\rho_A}(y, r) = \{ z \in \R^d \,:\, \rho_A (z-y) < r \}$
that contain $x$.

The significance of the Peetre-type maximal function in the following lemma
for our purposes is that it involves a mixture of the matrices $A$ and $B$,
in the sense that the convolution $f \ast \psi_j^B$ involves the matrix $B$,
whereas the weight $(1 + \rho_A (A^i z))^{\eta}$ involves the matrix $A$.
Its proof exploits the coarse equivalence of $A^*$ and $B^*$ in a crucial manner.

\begin{lemma}\label{lem:peetretype}
Suppose that $A^*$ and $B^*$ are coarsely equivalent.
With notation as in \Cref{sec:notation1}, for
$j \in \N_0$, $\eta > 0$ and $f \in \mathcal{S}'(\R^d)$, define
\[
  M^{\psi}_{j, \eta} f(x)
  := \max_{i \in I_j} \,\,
       \sup_{z \in \R^d} \,\,
         \frac{|(f \ast \psi^{B}_j)(x+z)|}{(1 + \rho_A (A^i z))^{\eta}},
  \quad x \in \R^d.
\]
Then there exists $C > 0$ (independent of $j, x, f$) such that
\[
  M_{j, \eta}^{\psi} f(x)
  \leq C  \bigg( M_{\rho_A} \big[|f \ast \psi_j^{B}|^{1/\eta} \big] (x) \bigg)^{\eta},
  \quad x \in \R^d,
\]
where $M_{\rho_A} $ denotes the Hardy-Littlewood maximal operator defined in \Cref{eq:HLMaximalOperator}.
\end{lemma}

\begin{proof}
Let $i \in I_j \subseteq \N_0$ be arbitrary.
Since $A^*$ and $B^*$ are coarsely equivalent, the associated covers $(Q^{A^*}_i)_{i \in \N_0}$
and $(P^{B^*}_j)_{j \in \N_0}$ from \Cref{sec:notation1} are equivalent by \Cref{lem:cover_coarse}.
Therefore, we see as in the proof of \Cref{cor:cover_coarse} (see \Cref{eq:EquivalenceImpliesNiceCovering})
that there exists $M \in \N$ (independent of $i,j$) such that
$\supp \widehat{\psi_j^B} = P_j^{B^*} \subseteq \bigcup_{\ell = - M}^M Q^{A^*}_{i + \ell}$.
Let
\[
  K
  := \bigcup_{\ell = - M}^M
       \overline{Q^{A^*}_{\ell}}
     \cup \bigcup_{\ell = -M}^M (A^*)^{\ell} \overline{Q}
  \qquad \text{and} \qquad
  K^* := \overline{\bigcup_{\ell = - \infty}^0 (A^*)^{\ell} K}.
\]
Note that $K \subseteq K^*$ and that $K,K^\ast$ are compact in $\R^d$ and do not depend on $i,j$.

Define $g := (f \ast \psi_j^B) \circ A^{-i}$.
Denoting the bilinear dual pairing between $\mathcal{S}'(\R^d)$ and $\mathcal{S}(\R^d)$
by $\langle \cdot , \cdot \rangle$, a direct calculation entails that,
for $\gamma \in \mathcal{S}(\R^d)$ with $\suppc \gamma \subseteq \R^d \setminus (A^\ast)^{-i} \overline{P_j^{B^\ast}}$,
\begin{align*}
 \langle \widehat{g} , \gamma \rangle
 &= \langle \widehat{f \ast \psi_j^B}, \gamma \circ (A^*)^{-i} \rangle
  = \langle \widehat{f}, \widehat{\psi_j^B} \cdot (\gamma \circ (A^*)^{-i}) \rangle
  = 0,
\end{align*}
and thus
\(
  \suppc \widehat{g}
  \subseteq (A^*)^{-i} \overline{P_j^{B^\ast}}
  \subseteq \bigcup_{\ell = -M}^M (A^*)^{-i} \overline{Q_{i + \ell}^{A^*}}
  .
\)
Note for $-M \leq \ell \leq M$ that if $i + \ell \leq M$, then $\overline{Q^{A^*}_{i + \ell}} \subseteq K$
and thus $(A^*)^{-i} \overline{Q^{A^*}_{i + \ell}} \subseteq K^*$.
On the other hand, $i+\ell > M$ for $-M \leq \ell \leq M$ implies $i > 0$ and
\[
  (A^*)^{-i} \overline{Q^{A^*}_{i+\ell}}
  = (A^\ast)^{-i} (A^\ast)^{i + \ell} \overline{Q}
  = (A^*)^{\ell} \overline{Q}
  \subseteq K
  \subseteq K^*
  .
\]
Overall, this shows that $\suppc \widehat{g} \subseteq K^*$.
An application of the anisotropic Peetre inequality (cf.\ \cite[Lemma 3.4]{bownik2006atomic})
therefore yields a constant $C = C(K^*, \eta) > 0$ such that
\begin{align}
  \sup_{z \in \R^d} \frac{|g(x-z)|}{(1 + \rho_A (z))^{\eta}}
  \leq C  [ (M_{\rho_A} |g|^{1/\eta}) (x) ]^{\eta},
  \quad x \in \R^d.
\end{align}
In view of the identity $M_{\rho_A} [h \circ A^k ] = (M_{\rho_A} h) \circ A^k$
for $h : \R^d \to \mathbb{C}$ and $k \in \Z$
(see, e.g., \cite[Lemma 3.1]{koppensteiner2023anisotropic1})
and since $\rho_A (-x) = \rho_A(x)$, this finally implies that
\begin{align*}
 \sup_{z \in \R^d} \frac{|(f \ast \psi_j^B)(x+z)|}{(1+ \rho_A(A^i z))^{\eta}}
 & = \sup_{z \in \R^d} \frac{|g(A^i(x + z))|}{(1+ \rho_A(A^i z))^{\eta}} \\
 & = \sup_{w \in \R^d} \frac{|g(A^i x - w)|}{(1+ \rho_A(w))^{\eta}} \\
 & \leq C  \big[ \big(M_{\rho_A} |g|^{1/\eta} \big) (A^i x) \big]^{\eta} \\
 & = C  \big[ \big(M_{\rho_A} (|g|^{1/\eta} \circ A^i) \big) (x) \big]^{\eta} \\
 & = C  \big[ \big(M_{\rho_A} |f \ast \psi_j^B|^{1/\eta}  \big) (x) \big]^{\eta}.
\end{align*}
Since $i \in I_j$ was chosen arbitrarily, this completes the proof.
\end{proof}

\subsection{Sufficient condition}

The following proposition is the main result of this section,
and settles the sufficient condition of \Cref{thm:main}.

\begin{proposition} \label{prop:sufficient}
  Suppose $A^*$ and $B^*$ are coarsely equivalent.
  Then $\TLA = \TLB$ for all $\alpha \in \mathbb{R}$, $p \in (0,\infty)$ and $q \in (0,\infty]$.
\end{proposition}

\begin{proof}
We will use the notation introduced in \Cref{sec:notation1}.
Let $\alpha \in \R$, $p \in (0,\infty)$, and $q \in (0,\infty]$.
We only show that
\(
  \| \cdot \|_{\mathbf{F}^{\alpha}_{p,q}(A; \varphi, \Phi)}
  \lesssim \| \cdot \|_{\mathbf{F}^{\alpha}_{p,q}(B; \psi, \Psi)}
  ;
\)
the reverse inequality follows by symmetry.
Throughout, fix some $\eta > \max \{1/p, 1/q\}$
and let $f \in \mathcal{S}'(\R^d)$.
Since $A^*$ and $B^*$ are coarsely equivalent, it follows that
$\sup_{i \in \N_0} |J_i| < \infty$ and $ \sup_{j \in \N_0} |I_j| < \infty$
by \Cref{lem:cover_coarse}.
\\~\\
\textbf{Step 1.} \emph{(Pointwise estimate.)}
Let $i \in \N_0$.
Define $\psi_B^{(i)} := \sum_{j \in J_i} \psi^B_j$.
Then $\psi^{(i)}_B \in \mathcal{S} (\R^d)$, and $\psi_B^{(i)} \ast \varphi_i^A = \varphi^A_i$
by condition \ref{enu:AnalyzingPairPartitionOfUnity} for $\psi,\Psi$.
Therefore, for $x \in \R^d$,
\begin{align*}
 |(f \ast \varphi_i^A )(x)|
 &= |(f \ast (\psi_B^{(i)} \ast \varphi_i^A)) (x)| \\
 &\leq \sum_{j \in J_i}
         | (f \ast (\psi_j^B \ast \varphi_i^A)) (x)| \\
 &\leq \sum_{j \in J_i}
         \int_{\R^d}
           \frac{|(f \ast \psi_j^B)(x+y)|}{(1+\rho_A (A^i y))^{\eta}}
           \big(1+\rho_A(A^iy) \big)^{\eta} |\varphi^A_i (-y)|
         \; dy \\
 &\leq \sum_{j \in J_i}
         M_{j, \eta}^{\psi} f(x)
         \int_{\R^d} \big(1+\rho_A(A^iy) \big)^{\eta} |\varphi^A_i (-y)| \; dy ,
\end{align*}
where $M_{j, \eta}^{\psi} f(x)$ is defined as in \Cref{lem:peetretype}.
For estimating the integral on the right-hand side above, choose $N > 1 + \eta$.
Then, since $\varphi, \Phi \in \mathcal{S} (\R^d)$, and in view of
\mbox{\cite[Chapter 1, Lemma 3.2]{bownik2003anisotropic}},
there exists $C > 0$ such that
$\max \{ |\Phi(\cdot)|, |\varphi (\cdot)|\} \leq C  (1+\rho_A(\cdot))^{-N}$.
In addition, since $\eta - N < -1$, an application of \cite[Lemma 2.3]{koppensteiner2023anisotropic1}
yields that $\int_{\R^d} (1+\rho_A(x))^{\eta - N} \; dx < \infty$.
Therefore, if $i = 0$, the symmetry of $\rho_A$ gives
\[
 \int_{\R^d} \big(1+\rho_A(A^iy) \big)^{\eta} |\varphi^A_i (-y)| \; dy
 \leq C \int_{\R^d} \big(1+\rho_A(y) \big)^{\eta - N}  \; dy
 < \infty
 .
\]
 Similarly, if $i \in \N$, then the change-of-variable $x = A^i y$ gives
\begin{align*}
  \int_{\R^d} (1+\rho_A (A^i y))^{\eta} |\varphi_i^A(-y)| \; dy
  & = \int_{\R^d} (1+\rho_A(x))^{\eta} |\varphi(-x)| \; dx \\
  & \leq C \int_{\R^d} (1 + \rho_A (x))^{\eta -N} \; dx
    < \infty
  ,
\end{align*}
where the right-hand side is independent of $i$.
Therefore,
\begin{align} \label{eq:pointwise}
  |(f \ast \varphi_i^A )(x)|
  \lesssim \sum_{j \in J_i} M_{j, \eta}^{\psi} f(x), \quad x \in \R^d.
\end{align}
Since $A^\ast, B^\ast$ are coarsely equivalent,
\Cref{cor:cover_coarse} shows that
\[
  |\det A|^i
  = |\det A^\ast|^i
  \asymp |\det B^\ast|^j
  = |\det B|^j
\]
whenever $i \in I_j$ (equivalently, $j \in J_i$).
Hence, combining this with \eqref{eq:pointwise} gives
\begin{align} \label{eq:pointwise2}
|\det A|^{\alpha i} |(f \ast \varphi_i^A )(x)|
 \lesssim \sum_{j \in J_i} |\det B|^{\alpha j} M_{j, \eta}^{\psi} f(x)
\end{align}
for $x \in \R^d$, with implied constant independent of $i \in \N_0$.
\\~\\
\textbf{Step 2.} \emph{(Norm estimate for $q < \infty$.)}
This step establishes the desired (quasi)-norm estimate
for the case $q < \infty$.
Since $\sup_{i \in \N_0} |J_i| < \infty$ and $\sup_{j \in \N_0} |I_j| < \infty$,
it follows from \Cref{eq:pointwise2} that, for every $x \in \R^d$,
\begin{align*}
 \sum_{i \in \N_0} \big( |\det A|^{\alpha i} |(f \ast \varphi_i^A )(x)| \big)^q
 &\lesssim \sum_{i \in \N_0} \bigg(  \sum_{j \in J_i} |\det B|^{\alpha j} M_{j, \eta}^{\psi} f(x) \bigg)^q \\
 &\lesssim \sum_{i \in \N_0} \sum_{j \in J_i} \big(  |\det B|^{\alpha j} M_{j, \eta}^{\psi} f(x) \big)^q \\
 &= \sum_{j \in \N_0} \sum_{i \in I_j} \big(  |\det B|^{\alpha j} M_{j, \eta}^{\psi} f(x) \big)^q \\
 &\lesssim \sum_{j \in \N_0}  \big(  |\det B|^{\alpha j} M_{j, \eta}^{\psi} f(x) \big)^q \\
 &\lesssim \sum_{j \in \N_0}
           \bigg(
             M_{\rho_A} \big[ |\det B|^{\frac{\alpha j}{\eta}} |f \ast \psi_j^B|^{\frac{1}{\eta}} \big] (x)
           \bigg)^{\eta q},
\end{align*}
where the last inequality used \Cref{lem:peetretype}.
Since $\eta q, \eta p > 1$, the vector-valued Fefferman-Stein inequality
(see, e.g., \cite[Theorem 2.5]{bownik2006atomic}) is applicable, and yields
\begin{align*}
 \| f \|_{\TL (A; \varphi, \Phi)}
 & = \bigg\|
       \bigg(
         \sum_{i \in \N_0}
           \big( |\det A|^{\alpha i} | f \ast \varphi_i^A | \big)^q
       \bigg)^{1/q}
     \bigg\|_{L^p} \\
 & \lesssim \bigg\|
              \bigg(
                \sum_{j \in \N_0}
                  \bigg(
                    M_{\rho_A}
                    \big[
                      |\det B|^{\frac{\alpha j}{\eta}}
                      |f \ast \psi_j^B|^{\frac{1}{\eta}}
                    \big]
                  \bigg)^{\eta q}
              \bigg)^{\frac{1}{\eta q}}
            \bigg\|_{L^{\eta p}}^{\eta} \\
 & \lesssim \bigg\|
              \bigg(
                \sum_{j \in \N_0}
                  \bigg(
                    |\det B|^{\frac{\alpha j}{\eta}}
                    |f \ast \psi_j^B |^{\frac{1}{\eta}}
                  \bigg)^{\eta q}
              \bigg)^{\frac{1}{\eta q}}
            \bigg\|_{L^{\eta p}}^{\eta} \\
 & = \| f \|_{\TL (B; \psi, \Psi)},
\end{align*}
which completes the proof for the case $q < \infty$.
\\~\\
\textbf{Step 3.} \emph{(Norm estimate for $q = \infty$.)}
As in Step 2, combining \Cref{eq:pointwise2}
with $\sup_{i \in \N_0} |J_i|, \; \sup_{j \in \N_0} |I_j| < \infty$ and \Cref{lem:peetretype}, yields
\begin{align*}
 \sup_{i \in \N_0} |\det A|^{\alpha i} |(f \ast \varphi_i^A )(x)|
 &\lesssim \sup_{i \in \N_0}  \sum_{j \in J_i} |\det B|^{\alpha j} M_{j, \eta}^{\psi} f(x)  \\
 &\lesssim \sup_{j \in \N_0}    |\det B|^{\alpha j} M_{j, \eta}^{\psi} f(x) \\
 &\lesssim \sup_{j \in \N_0}
           \bigg(
             M_{\rho_A}
             \big[ |\det B|^{\frac{\alpha j}{\eta}} |f \ast \psi_j^B|^{\frac{1}{\eta}} \big]
             (x)
           \bigg)^{\eta }
\end{align*}
for $x \in \R^d$.
Since $\eta p, q > 1$, an application of the vector-valued Fefferman-Stein inequality gives
 \begin{align*}
 \| f \|_{\TL (A; \varphi; \Phi)}
 & \lesssim \bigg\|
              \bigg(
                \sup_{j \in \N_0}
                    M_{\rho_A}
                    \big[
                      |\det B|^{\frac{\alpha j}{\eta}}
                      |f \ast \psi_j^B|^{\frac{1}{\eta}}
                    \big]
                  \bigg)^{\eta}
            \bigg\|_{L^{p}} \\
 & =        \bigg\|
                \sup_{j \in \N_0}
                    M_{\rho_A}
                    \big[
                      |\det B|^{\frac{\alpha j}{\eta}}
                      |f \ast \psi_j^B|^{\frac{1}{\eta}}
                    \big]
            \bigg\|_{L^{\eta p}}^\eta \\
 & \lesssim \bigg\|
                \sup_{j \in \N_0}
                  \bigg(
                    |\det B|^{\frac{\alpha j}{\eta}}
                    |f \ast \psi_j^B|^{\frac{1}{\eta}}
                  \bigg)
            \bigg\|_{L^{\eta p}}^\eta \\
  & =       \bigg\|
                \sup_{j \in \N_0}
                \bigg(
                  |\det B|^{\frac{\alpha j}{\eta}}
                  |f \ast \psi_j^B |^{\frac{1}{\eta}}
                \bigg)^{\eta}
            \bigg\|_{L^{p}} \\
  &= \| f \|_{\TL (B; \psi, \Psi)}.
\end{align*}
This completes the proof.
\end{proof}

\section{Necessary conditions for classification}
\label{sec:necessary}

In this section, we prove the necessary conditions of \Cref{thm:main}
for the equality of inhomogeneous Triebel-Lizorkin spaces.
Explicitly, we prove the following theorem.

\begin{theorem} \label{thm:necessary}
Let $A, B \in \GL(d, \R)$ be expansive matrices.
Suppose that $\TL(A) = \TL(B)$ for some $\alpha \in \R$, $p \in (0, \infty)$ and $q \in (0, \infty]$.
Then at least one of the following two cases hold:
\begin{enumerate}
 \item[(i)] $A^*$ and $B^*$ are coarsely equivalent;
 \item[(ii)] $\alpha = 0$, $p \in (1, \infty)$ and $q = 2$.
\end{enumerate}
\end{theorem}

\begin{remark}
In addition to \Cref{thm:necessary}, one can also show that if $\TLone(A) = \TLtwo(B)$
for some $\alpha, \beta \in \R$, $p_1, p_2 \in (0,\infty)$ and $q_1, q_2 \in (0, \infty]$,
then $\alpha = \beta$, $p_1 = p_2$ and $q_1 = q_2$.
This follows without much modification from the corresponding arguments
for the homogeneous Triebel-Lizorkin spaces in \cite[Section 5]{koppensteiner2023classification},
together with their adaptions to inhomogeneous function spaces that are
proven in this section.
As no new ideas are required, we do not provide the details.
\end{remark}

\subsection{General notation}
\label{sec:notation2}

Throughout all of this section, the same notation as in \Cref{sec:notation1} will be used.
In addition, define the index sets
\[
  N_i (A^*) := \{ k \in \N_0 : Q^{A^*}_i \cap Q^{A^*}_k \neq \emptyset \}
  \quad \text{and} \quad
  N_j(B^*) := \{ k \in \N_0 : P^{B^*}_k \cap P^{B^*}_j \neq \emptyset \}
\]
for fixed $i,j \in \N_0$.
Then a combination of  \Cref{cor:cover_coarse} and \Cref{lem:cardinality_uniform} (applied to $A=B$)
implies the existence of a constant $N \in \N$  such that
\[
  N_i(A^*) \cup N_i (B^*) \subseteq \{ k \in \N_0 : |k - i | \leq N \}
  \qquad \text{for all } i \in \N_0
  .
\]
For $i,j \in \N_0$, define the functions $\varphi_A^{(i)}, \psi_B^{(j)} \in \mathcal{S}(\R^d)$ by
\[
  \varphi_A^{(i)} :=  \sum_{k \in N_i(A^*)} \varphi^A_k
  \quad \text{and} \quad
  \psi^{(j)}_B := \sum_{k \in N_j(B^*) } \psi^B_k.
\]
Then, by condition \ref{enu:AnalyzingPairPartitionOfUnity},
it follows that $\widehat{\varphi_A^{(i)}} \equiv 1$ on $Q^{A^*}_i$,
and $\widehat{\psi_B^{(j)}} \equiv 1$ on $P^{B^*}_j$.

Lastly, we fix some $\chi \in \mathcal{S} (\R^d) \setminus \{0\}$
with the property that $\widehat{\chi} \geq 0$ and $\suppc \widehat{\chi} \subseteq \mathcal{B}(0, 1)$.
For $\delta > 0$, the associated (scalar) dilation of $\chi$ is defined by
$\chi_{\delta} := \delta^d  \chi(\delta \, \cdot)$.

\subsection{Auxiliary results}

This section contains two lemmata that are repeatedly in the remainder.

\begin{lemma}\label{lem:NormWithStarredPartitionOfUnity}
  Let $\alpha \in \R$, $p \in (0,\infty)$, and $q \in (0,\infty]$.
  With $\varphi_A^{(i)}$, $i \in \N_0$ as in \Cref{sec:notation2}, there exists a constant
  $C = C(\alpha, p, q, A, \varphi, \Phi) > 0$ satisfying
  \[
    \bigg\|
      \Big\|
        \Big(
          |\det A|^{\alpha i}
           |f \ast \varphi_A^{(i)} |
        \Big)_{i \in \N_0}
      \Big\|_{\ell^q}
    \bigg\|_{L^p}
    \leq C \| f \|_{\TL (A)}
  \]
  for all $f \in \Schwartz '(\R^d)$.
\end{lemma}

\begin{proof}
  We only provide the proof for $q < \infty$; the proof for $q = \infty$ is similar, but easier.
  With $N$ as in \Cref{sec:notation2}, it follows that for each $i \in \N_0$, we can write
  $N_i(A^\ast) = \{ \ell_1^{(i)}, \dots, \ell_{M_i}^{(i)} \}$
  with $M_i = |N_i (A^\ast)| \leq 2 N + 1$.
  Thus,
  \(
    \varphi_A^{(i)}
    = \sum_{t = 1}^{2N+1}
        \Indicator_{t \leq M_i} \,
        \varphi_{\ell_t^{(i)}}^A
    ,
  \)
  with $\Indicator_{t\leq M_i} = 1$ for $t\leq M_i$ and $\Indicator_{t\leq M_i} = 0$, otherwise.
  Hence, given $f \in \Schwartz '(\R^d)$,
  \[
    |f \ast \varphi_A^{(i)}|
    \leq \sum_{t = 1}^{2N+1}
         \big(
           \Indicator_{t \leq M_i}
           \cdot |f \ast \varphi_{\ell_t^{(i)}}^A |
         \big)
    .
  \]
  Furthermore, note because of $|\ell_t^{(i)} - i| \leq N$
  that $|\det A|^{\alpha i} \lesssim |\det A|^{\alpha \ell_t^{(i)}}$.
  Overall, this implies
  \begin{align*}
    \sum_{i \in \N_0}
      \big( 
        |\det A|^{\alpha i}
         |f \ast \varphi_A^{(i)}|
      \big)^q
    & \lesssim \sum_{i \in \N_0}
                 \sum_{t=1}^{2N+1}
                 \Big(
                   \Indicator_{t \leq M_i}
                   \cdot \big( 
                           |\det A|^{\alpha i}
                            |f \ast \varphi_{\ell_t^{(i)}}^A |
                         \big)^q
                 \Big) \\
    & \lesssim \sum_{i \in \N_0}
                 \sum_{t=1}^{2N+1}
                 \Big(
                   \Indicator_{t \leq M_i}
                   \cdot \big( 
                           |\det A|^{\alpha \ell_t^{(i)}}
                            |f \ast \varphi_{\ell_t^{(i)}}^A |
                         \big)^q
                 \Big)
    .
  \end{align*}
  Fix $\ell \in \N_0$ for the moment, and note that if $\ell = \ell_t^{(i)}$
  for some $i \in \N_0$ and $1 \leq t \leq M_i$, then $|\ell - i| = |\ell_t^{(i)} - i| \leq N$.
  Since also $M_i \leq 2 N + 1$, this implies that
  \[
    \# \{ (i,t) \,:\, i \in \N_0, 1 \leq t \leq M_i \text{ and } \ell_t^{(i)} = \ell \}
    \leq (2N+1)^2
    .
  \]
  Thus, in combination with the above, it follows that
  \begin{align*}
    \sum_{i \in \N_0}
      \big( 
        |\det A|^{\alpha i}
         |f \ast \varphi_A^{(i)}|
      \big)^q
    &\lesssim \sum_{i \in \N_0}
                 \sum_{t=1}^{2N+1}
                 \Big(
                   \Indicator_{t \leq M_i}
                   \cdot \big(
                           |\det A|^{\alpha \ell_t^{(i)}}
                            |f \ast \varphi_{\ell_t^{(i)}}^A |
                         \big)^q
                 \Big) \\
    &\lesssim \sum_{\ell \in \N_0}
               \big( 
                 |\det A|^{\alpha \ell}
                  |f \ast \varphi_{\ell}^A |
               \big)^q.
  \end{align*}
  By definition of $\| \cdot \|_{\TL (A)}$, this easily implies the claim.
\end{proof}

The following lemma is a consequence of the closed graph theorem.
We provide its proof for the sake of completeness.

\begin{lemma}\label{lem:NormEquivalence}
  Let $A,B \in \GL(d,\R)$ be expansive and let $\alpha \in \R$, $p \in (0,\infty)$,
  and $q \in (0,\infty]$.
  If $\TLA = \TLB$, then $\| \cdot \|_{\TLA} \asymp \| \cdot \|_{\TLB}$.
\end{lemma}

\begin{proof}
Suppose that $\TLA = \TLB$ as sets.
Then the identity map
\[
  \iota : \quad
  \TLA \to \TLB, \quad
  f \mapsto f
\]
is well-defined and linear.
Moreover, its graph is closed because if $f_n \to f$ in $\TLA$ and $f_n \to g$ in $\TLB$,
then \Cref{lem:Completeness} shows for arbitrary $\phi \in \Fourier (C_c^\infty (\R^d))$ that
\[
  \langle f, \phi \rangle
  = \lim_{n \to \infty} \langle f_n , \phi \rangle
  = \langle g , \phi \rangle
  .
\]
Note that $\Fourier(C_c^\infty(\R^d)) \subseteq \Schwartz(\R^d)$ is dense by
\cite[Theorems 7.7 and 7.10]{RudinFA}.
Hence, since $f,g \in \Schwartz' (\R^d)$, we get $f = g$,
showing that $\iota$ has closed graph. Therefore, it follows that $\| f \|_{\TLB} \lesssim \| f \|_{\TLA}$ by an application of the closed graph theorem
(see, e.g., \cite[Theorem 2.15]{RudinFA}), which is applicable
since $\TLA,\TLB$ are complete with respect to the quasi-norms
$\| \cdot \|_{\TLA}$ and $\| \cdot \|_{\TLB}$, which are $r$-norms for $r := \min \{ 1, p, q \}$,
cf.\ \Cref{lem:Completeness}.
This implies that the topology on $\TLA$ is induced by the complete, translation-invariant
metric $d(f,g) := \| f - g \|_{\TLA}^r$, and similarly for $\TLB$;
thus, $\TLA, \TLB$ are both F-spaces in the terminology of \cite[Section 1.8]{RudinFA}.

The estimate $\| f \|_{\TLB} \lesssim \| f \|_{\TLA}$ follows by symmetry.
\end{proof}

\subsection{The case \texorpdfstring{$\alpha \neq 0$}{α ≠ 0}}

This section is devoted to proving the necessary condition of \Cref{thm:main} for the case $\alpha \neq 0$.
A crucial ingredient in the proof of this result is the following proposition,
which is an adaptation of \cite[Proposition 5.3]{koppensteiner2023classification}
to the case of inhomogeneous function spaces.

\begin{proposition} \label{prop:normestimate1}
Let $\alpha \in \R$, $p \in (0,\infty)$ and $q \in (0,\infty]$.
If $f \in \mathcal{S}(\R^d)$ satisfies $\supp \widehat{f} \subseteq Q^{A^*}_{i_0}$
for some $i_0 \in \mathbb{N}_0$, then
\[
 \| f \|_{\TLA} \asymp | \det A|^{\alpha i_0} \| f \|_{L^p},
\]
with implicit constants independent of $i_0$ and $f$.
\end{proposition}

\begin{proof}
 Let $f \in \mathcal{S}(\R^d)$ be such that $\supp \widehat{f} \subseteq Q^{A^*}_{i_0}$ for $i_0 \in \N_0$.
 Then, using that $\supp \widehat{\varphi_i^A} = Q_i^{A^*}$ for $i \in \N_0$,
 we see that $f \ast \varphi_i^A = 0$ whenever $i \notin N_{i_0}(A^*)$.
 Therefore,
 \begin{equation} \label{eq:TL_sumLp}
   \begin{split}
     \| f \|_{\TL}
     & = \bigg\|
           \bigg(
             \sum_{i \in N_{i_0} (A^*)}
               \big(|\det A |^{\alpha i} |f \ast \varphi_i^A | \big)^q
           \bigg)^{1/q}
         \bigg\|_{L^p} \\
     & \lesssim_{p,q,N} \sum_{i \in N_{i_0} (A^*)}
                          |\det A|^{\alpha i} \| f \ast \varphi_i^A \|_{L^p},
   \end{split}
 \end{equation}
 with the usual modification in case of $q = \infty$.

 For further estimating the right-hand side above, note that  an application of Young's inequality
 implies that $\| f \ast \varphi_i^A \|_{L^p} \lesssim_{\varphi} \| f \|_{L^p}$
 provided that $p \in [1,\infty)$. For the case $p \in (0,1)$, note first that
 \begin{equation}
   \supp \widehat{f} , \; \supp \widehat{\varphi_i^A}
   \subseteq \bigcup_{\ell = - N}^N Q_{i_0 + \ell}^{A^*}
   \subseteq (A^*)^{i_0} K^\ast
   ,
   \label{eq:NormEstimate1SupportInclusion}
 \end{equation}
 where $K := \bigcup_{\ell = -N}^N (A^*)^{\ell} (\overline{Q} \cup \overline{Q_0})$
 and $K^\ast := \overline{\bigcup_{\ell = -\infty}^0 (A^\ast)^\ell K}$
 are compact and independent of $i_0, i$.
 To show that the second inclusion in \eqref{eq:NormEstimate1SupportInclusion} is indeed true,
 we distinguish two cases:
 In case of $i_0 + \ell \leq N$, we see because of
 $i_0 + \ell \geq \ell \geq -N$ that $Q_{i_0 + \ell}^{A^\ast} \subset K$,
 and thus
 \(
   Q_{i_0 + \ell}^{A^\ast}
   = (A^\ast)^{i_0} (A^\ast)^{-i_0} Q_{i_0 + \ell}^{A^\ast}
   \subset (A^\ast)^{i_0} K^\ast
   .
 \)
 If $i_0 + \ell > N$, then necessarily $i_0 > 0$, and thus
 \(
   Q_{i_0 + \ell}^{A^\ast}
   = (A^\ast)^{i_0 + \ell} Q
   = (A^\ast)^{i_0} (A^\ast)^{\ell} Q
   \subset (A^\ast)^{i_0} K
   \subset (A^\ast)^{i_0} K^\ast
   .
 \)
 In view of \eqref{eq:NormEstimate1SupportInclusion},
 choosing $R > 0$ such that $K^\ast \subseteq \mathcal{B}(0,R)$,
 an application of the convolution relation \cite[Theorem 3.4]{voigtlaender2023embeddings}
 (see also \cite[Section 1.5.1]{triebel2010theory}) yields that
 \begin{align*}
  \| f \ast \varphi_i^A \|_{L^p}
  & \leq [\Lebesgue{(A^*)^{i_0} \mathcal{B}(0,2R)}]^{\frac{1}{p} - 1}
         \| f \|_{L^p} \| \varphi_i^A \|_{L^p} \\
  & \lesssim_{A,\varphi, \Phi, N, p} |\det A|^{(i_0-i) (\frac{1}{p} - 1)}
                                     \| f \|_{L^p} \\
  & \lesssim_{A, N, p} \| f \|_{L^p}.
 \end{align*}

 Thus, $\| f \ast \varphi_i^A \|_{L^p} \lesssim \| f \|_{L^p}$ for all $|i_0 - i | \leq N$
 and all $p \in (0,\infty]$.
 Using this estimate in \eqref{eq:TL_sumLp} gives
 \[
   \| f \|_{\TL}
   \lesssim \sum_{i \in N_{i_0} (A^*)} |\det A |^{\alpha i} \| f \ast \varphi_i^A \|_{L^p}
   \lesssim |\det A |^{\alpha i_0} \| f \|_{L^p},
 \]
 with implicit constants independent of $i_0$ and $f$.

 For the reverse inequality, we use \Cref{lem:NormWithStarredPartitionOfUnity}
 and note that $f = f \ast \varphi_A^{(i_0)}$; thus,
 \[
   \| f \|_{\TL}
   \gtrsim |\det A|^{\alpha i_0}  \| f \ast \varphi_A^{(i_0)} \|_{L^p}
   =       |\det A|^{\alpha i_0}  \| f \|_{L^p}
   .
 \]
 This completes the proof.
\end{proof}

Using \Cref{prop:normestimate1}, we now prove the necessity in \Cref{thm:main}
for the case $\alpha \neq 0$.

\begin{theorem} \label{thm:casenonzero}
  Suppose that $\TL(A) = \TL(B)$ for some $\alpha \in \mathbb{R} \setminus \{0\}$,
  $p \in (0,\infty)$ and $q \in (0,\infty]$.
  Then $A^*$ and $B^*$ are coarsely equivalent.
\end{theorem}

\begin{proof}
Let $i, j \in \N_0$ be arbitrary with $Q_{i}^{A^*} \cap P_{j}^{B^*} \neq \emptyset$.
Choose $\xi_0 \in \R^d$ and $\delta > 0$ such that
$\mathcal{B}(\xi_0, \delta) \subseteq Q_{i}^{A^*} \cap P_{j}^{B^*}$,
which is possible since $Q_{i}^{A^*} , P_{j}^{B^*}$ are open.
Define $f^{(\delta)} := M_{\xi_0}  \chi_{\delta}$, where $\chi$ is as in Section~\ref{sec:notation2}.
Then it follows that
\(
  \supp \widehat{f^{(\delta)} }
  \subseteq \mathcal{B}(\xi_0, \delta)
  \subseteq Q_i^{A^*} \cap P_j^{B^*}
  .
\)
Hence, applying \Cref{prop:normestimate1} to $f^{(\delta)}$ (with $A$ and $B$) gives
\[
  |\det A|^{\alpha i} \delta^{d(1-\frac{1}{p})}
  \asymp |\det A|^{\alpha i} \| f^{(\delta)} \|_{L^p}
  \asymp \| f^{(\delta)} \|_{\TL(A)}
  \asymp \| f^{(\delta)} \|_{\TL(B)} \|
  \asymp |\det B|^{\alpha j} \delta^{d(1-\frac{1}{p})},
\]
where we also used \Cref{lem:NormEquivalence}.
Note that the implicit constants are independent of $i,j$.
Thus, canceling the factor involving $\delta$, we see that there exists a constant $C > 0$
(independent of $i,j$) such that
\[
  \frac{1}{C} |\det A^\ast|^{\alpha i}
  \leq |\det B^\ast|^{\alpha j}
  \leq C |\det A^\ast|^{\alpha i}
  \quad \text{for all $i,j \in \N_0$ for which}
  \quad Q_{i}^{A^*} \cap P_{j}^{B^*} \neq \emptyset.
\]
Since $\alpha \neq 0$, an application of \Cref{lem:cardinality_uniform} therefore yields a constant
$M \in \mathbb{N}$ such that
\[
  J_i \subseteq \big\{ j \in \N_0 : | j - \lfloor \varepsilon i  \rfloor | \leq M \big\}
  \qquad \text{and} \qquad
  I_j \subseteq \Big\{ i \in \N_0 : \Big| i - \Big\lfloor  \frac{j}{\varepsilon}   \Big\rfloor \Big| \leq M \Big\}
\]
for all $i, j \in \N_0$, where $\varepsilon := \ln |\det A| / \ln |\det B|$.
In particular, this implies that $|J_i|, |I_j| \lesssim 1$ with implicit constant independent of $i, j \in \N_0$.
Thus, $A^*$ and $B^*$ are coarsely equivalent by \Cref{lem:cover_coarse}.
\end{proof}

\subsection{The case \texorpdfstring{$\alpha = 0$}{α = 0} and \texorpdfstring{$q \neq 2$}{q ≠ 2}}

This subsection is concerned with proving the necessary condition for the case  $\alpha = 0$ and $q \neq 2$.
For this, we need in addition to \Cref{prop:normestimate1} the following more refined version.

\begin{proposition} \label{prop:normestimate2}
Let  $\alpha \in \R$, $p \in (0, \infty)$ and $q \in (0,\infty]$.
For $K \in \N$, let $(i_k)_{k = 1}^K$ be a sequence in $\N_0$ such that $|i_k - i_{k'}| > 2N$
for $k \neq k'$, where $N \in \N$ is the constant fixed in \Cref{sec:notation2}.
Let $\chi$ be as in \Cref{sec:notation2}.

If there exist $\delta > 0$ and points $\xi_1, ..., \xi_K \in \R^d$ such that
\[
  \mathcal{B}(\xi_k, \delta) \subseteq Q_{i_k}^{A^*},
  \quad \text{for all} \quad k = 1, ..., K,
\]
then, for any $c \in \mathbb{C}^K$, the function $f := \sum_{k = 1}^K c_k M_{\xi_k} \chi_{\delta}$
satisfies the norm estimate
\begin{equation}
  \| f \|_{\TL(A)}
  \asymp \delta^{d(1-1/p)} \bigg\| \big(|\det A |^{\alpha i_k} |c_k |\big)_{k = 1}^K \bigg\|_{\ell^q},
  \label{eq:NormEstimate2}
\end{equation}
with implicit constants independent of $K, c, \delta, (\xi_k)_{k = 1}^K$ and $(i_k)_{k = 1}^K$.
\end{proposition}

\begin{proof}
We only deal with the case $q < \infty$; the case $q = \infty$ follows by the usual modification.
The proof follows (parts of) the arguments proving \cite[Proposition 5.5]{koppensteiner2023classification} closely.

Throughout, let $\delta, (\xi_k)_{k = 1}^K,(i_k)_{k = 1}^K$, and $f$ be as in the statement of the proposition.
Then, since $\mathcal{B}(\xi_k, \delta) \subseteq Q_{i_k}^{A^*}$, it follows that
$\supp \widehat{M_{\xi_k} \chi_\delta} = \supp T_{\xi_k} \widehat{\chi_{\delta}} \subseteq Q^{A^*}_{i_k}$ for $k = 1, ..., K$.
On the other hand, $\supp \widehat{\varphi_i^A} = Q_i^{A^*}$ for $i \in \N_0$.
Therefore, $M_{\xi_k} \chi_{\delta} \ast \varphi_i^A = 0$ whenever $|i - i_k | > N$
as then $i \notin N_{i_k}(A^*)$.
Since, for fixed $i \in \N_0$, there can be at most one $i_k$ such that $|i - i_k|\leq N$,
it follows that
\[
  f \ast \varphi_i^A
  = \sum_{\ell = 1}^K c_\ell \cdot (M_{\xi_\ell} \chi_{\delta} \ast \varphi_i^A)
  = \begin{cases}
      c_k \cdot (M_{\xi_k} \chi_{\delta} \ast \varphi_i^A) , & \text{if } |i-i_k| \leq N \text{ for some } 1 \leq k \leq K \\
      0,                                                     & \text{if } |i - i_k| > N \text{ for all } 1 \leq k \leq K.
    \end{cases}
\]
Therefore, if $|i- i_k | \leq N$, we can estimate
\begin{align} \label{eq:penultimate}
  | f \ast \varphi_i^A (x)|
  & \leq |c_k| \cdot (|\chi_{\delta}| \ast| \varphi_i^A|) (x)
  \lesssim_{d, p, N, \varphi, \Phi, \chi, A} |c_k|  \delta^d  (1+|\delta x |)^{-\frac{d}{p} - 1},
\end{align}
where the last inequality follows from an application%
\footnote{The statement of \cite[Lemma A.3]{koppensteiner2023classification} assumes
  slightly different conditions on $\varphi$, but its proof is valid
for general Schwartz functions $\varphi \in \mathcal{S}(\R^d)$.}
of \cite[Lemma A.3]{koppensteiner2023classification}
(applied to the bounded set $Q \cup Q_0$,  $\ell = i_k$ and $M = d/p + 1$).
This, together with $|f \ast \varphi_i^A (x)| = 0$ for $|i - i_k| > N$, yields the estimate
\begin{align*}
  \bigg( \sum_{i \in \N_0} \big( |\det A|^{\alpha i} |f \ast \varphi^A_i (x) | \big)^q \bigg)^{1/q}
  &\leq \bigg(
          \sum_{k = 1}^K \,\,
            \sum_{\substack{i \in \N_0 \\ |i - i_k| \leq N}} \,\,
              \big(|\det A|^{\alpha i} |f \ast \varphi^A_i (x) | \big)^q
        \bigg)^{1/q} \\
  &\lesssim \bigg(
              \sum_{k = 1}^K 
              \big(
                |\det A|^{\alpha i_k}
                |c_k|
                 \delta^d
                 (1 + |\delta x|)^{-\frac{d}{p} - 1}
              \big)^q
            \bigg)^{1/q} \\
  &= \delta^d
      (1 + |\delta x|)^{-\frac{d}{p} - 1}
      \bigg\| \big(|\det A|^{\alpha i_k} |c_k| \big)_{k = 1}^K \bigg\|_{\ell^q},
\end{align*}
where the penultimate step uses Equation \eqref{eq:penultimate}
and $N_{i_k} (A^*) \lesssim_N 1$ for $k = 1, ..., K$.
Hence, taking the $L^p$-(quasi)-norm yields
\begin{align*}
  \| f \|_{\TL(A)}
  & \lesssim \bigg(
               \int_{\R^d} \big(\delta^d (1+|\delta x|)^{-\frac{d}{p} - 1} \big)^p \; dx
             \bigg)^{1/p}
             \bigg\| \big(|\det A|^{\alpha i_k} |c_k| \big)_{k = 1}^K \bigg\|_{\ell^q} \\
  & \lesssim_{d, p} \delta^{d(1-1/p)}
                    \bigg\|
                      \big(|\det A|^{\alpha i_k} |c_k| \big)_{k = 1}^K
                    \bigg\|_{\ell^q},
\end{align*}
which establishes one of the inequalities in \Cref{eq:NormEstimate2}.

For the reverse inequality, note first for $\varphi_A^{(i)}$ as in \Cref{sec:notation2} that 
\[
  \supp \widehat{\varphi^{(i_k)}_A} \subseteq \bigcup_{\ell = -N}^N Q^{A^*}_{i_k + \ell}
  \quad \text{for all } k = 1,\dots, K
  .
\]
The assumption $|i_{k'} - i_k | > 2N$ for $k \neq k'$,
together with $N_i(A^*) \subseteq \{j \in \N_0 : | i - j | \leq N \}$ for all $i \in \N_0$
(see \Cref{sec:notation2}), yields
\[
  Q^{A^*}_{i_{k}} \cap \bigcup_{\ell = -N}^N Q^{A^*}_{i_{k'} + \ell} = \emptyset,
  \quad k \neq k',
\]
and hence $M_{\xi_k} \chi_{\delta} \ast \varphi_A^{(i_{k'})} = 0$ for $k \neq k'$.
Additionally, $\widehat{\varphi_A^{(i_k)}} \equiv 1$ on $Q^{A^*}_{i_k}$, and thus
\[
  f \ast \varphi_A^{(i_k)} = c_k \cdot M_{\xi_k} \chi_{\delta}
\]
for all $k = 1, ..., K$.
Using this identity, together with \Cref{lem:NormWithStarredPartitionOfUnity},
a direct calculation entails
\begin{align*}
  \| f \|_{\TL(A)}
  \gtrsim \bigg\|
            \bigg(
              \sum_{k = 1}^K \big(|\det A|^{\alpha i_k} | f \ast \varphi_A^{(i_k)} | \big)^q
            \bigg)^{1/q}
          \bigg\|_{L^p}
  \geq \| \chi_{\delta} \|_{L^p}
       \bigg\| \big(|\det A|^{\alpha i_k} |c_k| \big)_{k = 1}^K \bigg\|_{\ell^q} .
\end{align*}
Since $\| \chi_{\delta} \|_{L^p} = \delta^{d(1-1/p)} \| \chi \|_{L^p}$, this finishes the proof.
\end{proof}

\begin{theorem} \label{thm:casenzero1}
  Let  $p \in (0, \infty)$ and $q \in (0, \infty]$.
  If $\TLzero(A) = \TLzero(B)$ and $A^*$ and $B^*$ are not coarsely equivalent, then $q = 2$.

  Consequently, if $\TLzero(A) = \TLzero(B)$ for some $q \neq 2$, then $A^*$ and $B^*$ are coarsely equivalent.
\end{theorem}

\begin{proof}
Suppose that $\TLzero(A) = \TLzero(B)$ and that $A^*$ and $B^*$ are not coarsely equivalent.
By \Cref{lem:cover_coarse}, the latter condition is equivalent to
$\sup_{i \in \N_0} |J_i| + \sup_{j \in \N_0} |I_j| = \infty$.
Throughout, we assume that $\sup_{j \in \N_0} |I_j | = \infty$, the other case being similar.
We split the proof into two steps.
\\~\\
\textbf{Step 1.}
In this step, we show that, for arbitrary $K \in \N$, there exist $\delta > 0$ and
$j_0 = j_0 (K) \in \N_0$, as well as sequences $(i_k)_{k = 1}^K \subseteq \N_0$
and $(\xi_k)_{k = 1}^K \subseteq \R^d$
satisfying the assumptions of \Cref{prop:normestimate2} and furthermore
$\CalB (\xi_k, \delta) \subseteq Q_{i_k}^{A^\ast} \cap P_{j_0}^{B^\ast}$.

Since $\sup_{j \in \N_0} |I_j | = \infty$, there exists $j_0 \in \N_0$
for which $|I_{j_0}| \geq (2N+1) K$, where $N \in \N$ is the fixed constant from \Cref{sec:notation2}.
For $n = 0, ..., 2N$, set $\N_0^{(n)} := n + (2N+1) \N_0$.
Then $I_{j_0} = \bigcup_{n = 0}^{2N} (\N_0^{(n)} \cap I_{j_0})$,
and hence there exists $n \in \{0, ..., 2N \}$ for which $|I_{j_0} \cap \N_0^{(n)} | \geq K$.
Thus, there exist pairwise distinct indices $i_1, ..., i_K \in I_{j_0} \cap \N_0^{(n)}$,
which then necessarily satisfy $|i_k - i_{k'} | \geq 2N+1$ for $k\neq k'$.
The intersections $Q^{A^*}_{i_k} \cap P^{B^*}_{j_0} \neq \emptyset$ being open
for each $k \in \{1, ..., K\}$, one can choose points $\xi_1, ..., \xi_K \in \R^d$
and a constant $\delta > 0$ such that
\begin{align} \label{eq:ball}
\mathcal{B}(\xi_k, \delta) \subseteq Q^{A^*}_{i_k} \cap P^{B^*}_{j_0}, \quad k = 1, ..., K,
\end{align}
 as required.
\\~\\
\textbf{Step 2.}
Let $K \in \N$, and let $\delta > 0$, $j_0 \in \N_0$,
as well as $(i_k)_{k = 1}^K$ and $(\xi_k)_{k = 1}^K$ be as in Step 1,
and let $c \in \mathbb{C}^K$ be arbitrary.
Given $\theta \in \{-1, +1\}^K$, define
\[
  f_{\theta, c}
  := \sum_{k = 1}^K \theta_k \, c_k \, M_{\xi_k} \, \chi_{\delta}
  \in \Schwartz(\R^d).
\]
If $\theta$ is considered as a random vector which is uniformly distributed in $\{ \pm 1 \}^K$
and denoting the expectation with respect to $\theta$ by $\EE_{\theta}$,
then an application of Khintchine's inequality
(see, e.g., \cite[Proposition 4.5]{WolffLecturesOnHarmonicAnalysis}) gives
\begin{align*}
  \mathbb{E}_{\theta} \| f_{\theta, c} \|_{L^p}^p
  & = \mathbb{E}_{\theta}
        \int_{\R^d}
          |\chi_{\delta} (x)|^p \;
          \bigg|
            \sum_{k = 1}^K
              \theta_k \,
              c_k \,
              e^{2 \pi i \xi_k \cdot x }
          \bigg|^p
        \; dx \\
  & = \int_{\R^d}
          |\chi_{\delta} (x)|^p \;
          \mathbb{E}_{\theta}
          \bigg|
            \sum_{k = 1}^K
              \theta_k \,
              c_k \,
              e^{2 \pi i \xi_k \cdot x }
          \bigg|^p
        \; dx \\
  & \asymp \int_{\R^d}
               |\chi_{\delta} (x)|^p \;
               \bigg( \sum_{k = 1}^K |c_k|^2 \bigg)^{p/2}
             \; dx \\
  & \asymp \delta^{d(p - 1)} \| c \|_{\ell^2}^p ,
    \numberthis \label{eq:Khintchine}
\end{align*}
with implied constants only depending on $p, d, \chi$.

We next apply \Cref{prop:normestimate1} and \Cref{prop:normestimate2} to $f_{\theta, c}$.
First, since
\[
  \supp \widehat{f_{\theta, c}}
  \subseteq \bigcup_{k=1}^K \CalB (\xi_{k}, \delta)
  \subseteq P^{B^*}_{j_0}
  ,
\]
an application of \Cref{prop:normestimate1} gives
\[
  \| f_{\theta, c} \|_{\TLzero(B)} \asymp \| f_{\theta, c} \|_{L^p}.
\]
On the other hand, an application of \Cref{prop:normestimate2} yields that
\[
  \| f_{\theta, c} \|_{\TLzero(A)} \asymp \delta^{d(1-1/p)} \| c \|_{\ell^q}.
\]
Since $\| f \|_{\TLzero(A)} \asymp \| f \|_{\TLzero(B)}$ by \Cref{lem:NormEquivalence},
a combination of these estimates yields that
$\delta^{d(1-1/p)} \| c \|_{\ell^q} \asymp \| f_{\theta, c} \|_{L^p}$ and hence
\[
  \delta^{d (p - 1)} \, \| c \|_{\ell^q}^p \asymp \| f_{\theta,c} \|_{L^p}^p
  .
\]
Combining this in turn with \Cref{eq:Khintchine} yields $\| c \|_{\ell^q}^p \asymp \| c \|_{\ell^2}^p$,
with implicit constants independent of $c$ and $K$.
Since $K \in \N$ and $c \in \mathbb{C}^K$ were chosen arbitrarily, this implies that $q = 2$.
\end{proof}

\subsection{The case \texorpdfstring{$\alpha = 0$}{α = 0} and \texorpdfstring{$q = 2$}{q = 2}}
\label{sec:NecessityQEqualTwo}

This final subsection treats the Triebel-Lizorkin spaces $\mathbf{F}^{0}_{p,2}(A)$ with $p \in (0,\infty)$.
By \Cref{prop:TL_hardy}, these spaces correspond to $\LHA = \mathbf{F}^{0}_{p,2}(A)$
for $p \in (0,1]$ and to $L^p = \mathbf{F}^{0}_{p,2}(A)$ for $p > 1$.
Hence, it remains to consider the case $p \in (0,1]$.

We start by introducing a family of functions that will be used
in the proof of \Cref{thm:casenzero2} below.
Let $A, B \in \mathrm{GL}(d, \R)$ be expansive matrices.
Fix $p \in (0,1]$ and let
\begin{align} \label{eq:sadmissible}
  s
  \geq \max
       \big\{
         \lfloor ( \tfrac{1}{p} - 1)  \zeta_-(A)^{-1}\rfloor, \,\,
         \lfloor( \tfrac{1}{p} - 1)  \zeta_-(B)^{-1}\rfloor
       \big\}.
\end{align}
We will consider the following conditions on a measurable function $f : \R^d \to \mathbb{C}$:
\begin{enumerate}
  \item[(f1)] \label{enu:BownikSpecialFunctionsSupportCondition}
              $\supp f \subseteq x_0 + B^{j_1} A^{j_2} \mathcal{B}(0,1)$
              for some $x_0 \in \R^d$ and $j_1 \in \mathbb{N}_0$ and $j_2 \in \Z$;

  \item[(f2)] \label{enu:BownikSpecialFunctionsLInftyBound}
              $\| f \|_{L^\infty} \leq |\det B|^{-{j_1}/p} |\det A|^{-{j_2}/p}$;

  \item[(f3)] \label{enu:BownikSpecialFunctionsVanishingMoments}
              $\int_{\R^d} f(x) x^{\sigma} \; dx = 0$ for all $\sigma \in \N_0^d$
              satisfying $|\sigma| \leq s$.
\end{enumerate}

An essential property of functions satisfying
\ref{enu:BownikSpecialFunctionsSupportCondition}-\ref{enu:BownikSpecialFunctionsVanishingMoments}
is given by the following lemma.
Its proof is more refined than corresponding results for (nonlocal) anisotropic Hardy spaces
(see, e.g., the proof of \cite[Chapter 1, Theorem 10.5]{bownik2003anisotropic})
due to the fact that dilations $D_A^p$ do generally not act isometrically on local Hardy spaces $\LHA$.
In addition, we need to consider $j_1 \geq 0$ in condition \ref{enu:BownikSpecialFunctionsSupportCondition}.

\begin{lemma} \label{lem:test_function_hardy}
Suppose $\LHA = \LHB$ for some $p \in (0,1]$.
Then there exists a constant $C > 0$ such that $\| f \|_{\LHA}, \| f \|_{\LHB} \leq C$
for all functions $f$ satisfying conditions {\normalfont 
\ref{enu:BownikSpecialFunctionsSupportCondition}-\ref{enu:BownikSpecialFunctionsVanishingMoments}}.
\end{lemma}

\begin{proof}
Recall that since $\LHA = \LHB$, it follows that $\| \cdot \|_{\LHA} \asymp \| \cdot \|_{\LHB}$
by a combination of \Cref{prop:TL_hardy} and \Cref{lem:NormEquivalence}.

Let $f$ satisfy \ref{enu:BownikSpecialFunctionsSupportCondition}-\ref{enu:BownikSpecialFunctionsVanishingMoments}.
Then the support of $D^{p}_{B^{j_1}} f$ is  $B^{-j_1} \supp f \subseteq B^{-j_1} x_0 + A^{j_2} \mathcal{B}(0,1)$.
Moreover, $D^{p}_{B^{j_1}} f$ satisfies the norm estimate
\[
  \| D^p_{B^{j_1}} f \|_{L^\infty}
  = |\det B|^{j_1 /p} \| f \|_{L^\infty}
  \leq |\det A|^{-j_2/p}.
\]
Finally, $\int_{\R^d} D^p_{B^{j_1}} f(x) x^{\sigma} \; dx = 0$ for all $|\sigma| \leq s$.
Thus, by \Cref{rem:alternative_atom}, the function $D^p_{B^{j_1}} f$ is (a constant multiple of)
a $(p,s)$-atom associated to $A$.
Therefore, by \mbox{\cite[Chapter 1, Theorem 4.2]{bownik2003anisotropic}},
it follows that $\| D^p_{B^{j_1}} f \|_{H^p (A)} \lesssim 1$,
with a constant independent of $j_1$ and $f$.

In view of the above and the assumption $\| \cdot \|_{\LHA} \asymp \| \cdot \|_{\LHB}$,
it remains to prove the estimate $\| f \|_{\LHB} \lesssim \| D^p_{B^{j_1}} f \|_{H^p (A)}$.
For this, note first that, for any measurable function $h : \R^d \to \mathbb{C}$ and any $x \in \R^d$,
\begin{align*}
  M^{0, \loc}_{\phi,B} [D^p_B h] (x)
  & = \sup_{j \in \mathbb{N}_0}
        |\det B|^j | ((D^p_B h) \ast (\phi \circ B^j) ) (x) | \\
  & = \sup_{j \in \mathbb{N}_0}
        |\det B|^{1/p}
        |\det B|^{j-1}
        |(h \ast (\phi \circ B^{j-1}) )(Bx)| \\
  & \geq |\det B|^{1/p}
          \sup_{\ell \in \mathbb{N}_0}
            |\det B|^{\ell}
            |(h \ast (\phi \circ B^{\ell}) )(Bx)| \\
  & =   |\det B|^{1/p} (M^{0, \loc}_{\phi, B} h)(Bx)
  .
\end{align*}
Hence,
\[
  \| h \|_{\LHB}
  = \| M_{\phi,B}^{0,\loc} h \|_{L^p}
  = |\det B|^{1/p} \, \| (M_{\phi,B}^{0,\loc} h) (B \cdot) \|_{L^p}
  \leq \| M_{\phi,B}^{0,\loc} [D^p_B h] \|_{L^p}
  = \| D^p_B h \|_{\LHB}
  ,
\]
which implies, in particular, that $\| f \|_{\LHB} \leq \| D^p_{B^{j_1}} f\|_{\LHB}$ since $j_1 \geq 0$.
Second, by definition, it holds that $H^p (A) \hookrightarrow \LHA$.
All in all, this gives
\[
  \| f \|_{\LHB}
  \leq \| D^p_{B^{j_1}} f \|_{\LHB}
  \lesssim \| D^p_{B^{j_1}} f \|_{\LHA}
  \lesssim \| D^p_{B^{j_1}} f \|_{H^p (A)} 
  \lesssim 1
  ,
\]
where the second inequality follows from $\| \cdot \|_{\LHA} \asymp \| \cdot \|_{\LHB}$.
\end{proof}

The following theorem provides the desired necessary condition for the equality
of aniso\-tropic local Hardy spaces associated to different expansive matrices $A,B$.
Its proof structure is analogous to the classification of anisotropic (nonlocal)
Hardy spaces in \cite{bownik2003anisotropic}, with various essential modifications; see also \Cref{rem:example}.

\begin{theorem} \label{thm:casenzero2}
If $\LHA = \LHB$ for some $p \in (0,1]$,
  then $A^*$ and $B^*$ are coarsely equivalent.
\end{theorem}

\begin{proof}
Arguing by contradiction, assume that $A^*$ and $B^*$ are not coarsely equivalent.
Then, by \Cref{lem:coarse_norm}, it follows for
\(
  \eps
  = \ln |\det A^\ast| / \ln |\det B^\ast|
  = \ln |\det A| / \ln |\det B|
\)
that
\[
  \sup_{k \in \N}
    \| B^{\lfloor \varepsilon k \rfloor} A^{-k} \|
  = \sup_{k \in \N}
      \| (A^\ast)^{-k} (B^\ast)^{\lfloor \varepsilon k \rfloor} \|
  = \infty.
\]
Hence, by passing to a subsequence if necessary, it may be assumed that
\[
  \lim_{k \to \infty} \| B^{\lfloor \varepsilon k \rfloor} A^{-k} \| = \infty.
\]
Let $d(k) \in \Z$  be minimal with the property that
$\| B^{\lfloor \eps k\rfloor} A^{-k - d(k)} \| \leq 1$.
Then, as in \cite[Chapter 1, Theorem 10.5]{bownik2003anisotropic},
it follows that
\(
  1
  < \| B^{\lfloor \eps k \rfloor} A^{-k - (d(k) - 1)} \|
  \leq \| B^{\lfloor \eps k \rfloor} A^{-k - d(k)} \| \cdot \| A \|
  ,
\)
and hence
\begin{equation}
  1
  \geq c(k)
  := \| B^{\lfloor \eps k\rfloor} A^{-k - d(k)} \|
  \geq \| A \|^{-1}.
  \label{eq:COfKDefinition}
\end{equation}
Moreover, we have $d(k) \to \infty$ as $k \to \infty$, which follows by recalling that
$\| B^{\lfloor \eps k\rfloor} A^{-k - d(k)} \| \leq 1$, and hence
\[
  \| A \|^{d(k)}
  \geq \| A^{d(k)} \|
  \geq \| B^{\lfloor \eps k\rfloor} A^{-k - d(k)} \| \cdot \| A^{d(k)} \|
  \geq \| B^{\lfloor \eps k\rfloor} A^{-k} \|
  \to \infty
\]
as $k \to \infty$.

In order to simplify notation, denote
\[
  Q_k := B^{\lfloor \eps k\rfloor} A^{-k - d(k)}
\]
and let $z_k \in \R^d$ be such that
\[
  |z_k|=1
  \qquad \text{and} \qquad
  |Q_k z_k| = \|Q_k\| = c(k) 
  .
\]
In addition, let $U_k \in \R^{d\times d}$ be an orthogonal matrix satisfying $U_k e_1 = z_k$,
where $e_1$ denotes the first element of the canonical basis for $\R^d$.
Using the matrices $Q_k$ and $U_k$ for $k \in \N$, we define the sequence of functions
\[
  f_k := D^p_{Q_k^{-1}} \, D^p_{U_k^{-1}} \, f_0,
\]
where $f_0 : \R^d \to \mathbb{C}$ is a bounded measurable function satisfying
\begin{align} \label{eq:functionf0}
  f_0 (x) =
  \begin{cases}
    \delta_0 > 0, & \text{if } x \in \mathcal{B}(\frac{3}{4} e_1, \frac{1}{4}) \\
    0,            & \text{if } x \notin \mathcal{B}(0, \frac{1}{2}) \cup \mathcal{B}(\frac{3}{4} e_1, \frac{1}{4})
  \end{cases}
\end{align}
and such that conditions
\ref{enu:BownikSpecialFunctionsSupportCondition}--\ref{enu:BownikSpecialFunctionsVanishingMoments}
 hold with $x_0 = 0$ and $j_1 = j_2 = 0$.
The existence of such a function is guaranteed by \Cref{lem:SpecialFunctionExists}.
It is then not hard to see that also each function $f_k$, $k \in \N$,
satisfies conditions \ref{enu:BownikSpecialFunctionsSupportCondition}--\ref{enu:BownikSpecialFunctionsVanishingMoments}
with $x_0 = 0$, $j_1 = \lfloor \eps k \rfloor$ and $j_2 = - k - d(k)$.

The remainder of the proof is split into two steps,
which consider the cases $p < 1$ and $p = 1$ separately.
\\~\\
\textbf{Step 1.} \emph{(Case $p < 1$)}.
In this step, we show that $\|f_k\|_{\LHB} \to \infty$ as $k \to \infty$.
Since $\| f_k \|_{\LHB} \lesssim 1$ by \Cref{lem:test_function_hardy},
this will provide the desired contradiction.

Since $Q_k U_k \mathcal{B}(0, \frac{1}{2}) \subseteq \mathcal{B}(0, \frac{c(k)}{2})$ and
\(
  Q_k U_k \mathcal{B}(\frac{3}{4} e_1, \frac{1}{4})
  = Q_k \mathcal{B}(\frac{3}{4} z_k, \frac{1}{4})
  \subseteq \CalB (\frac{3}{4} Q_k z_k, \frac{1}{4}),
\)
it follows by the definition of $f_k$ and \eqref{eq:functionf0}
that if $f_k (x) \neq 0$ for $x \in \R^d \setminus \mathcal{B}(0, \frac{c(k)}{2})$, then
\begin{align} \label{eq:valuefk}
  f_k (x)
  = \delta_0  |\det B|^{-\frac{\lfloor \varepsilon k \rfloor}{p}} |\det A|^{\frac{k + d(k)}{p}}
  =: \delta_k,
  \qquad \text{and} \qquad
  x \in Q_k \mathcal{B}(\tfrac{3}{4} z_k , \tfrac{1}{4}).
\end{align}
Let $\phi \in \mathcal{S}(\R^d)$ be a fixed nonnegative Schwartz function satisfying
$\phi \equiv 1$ on $\mathcal{B}(0, \frac{1}{8} \| A \|^{-1})$
and $\phi \equiv 0$ outside of $\mathcal{B}(0, \frac{3}{16} \|A\|^{-1})$.
Then, for $z \in \R^d$,
\begin{align} \label{eq:maximal_lower}
   M^{0, \loc}_{\phi, B} f_k (z)
   \geq | f_k \ast \phi (z)|
   = \bigg| \int_{\R^d} f_k (x) \phi (z - x) \; dx \bigg|.
 \end{align}
Fix
\(
  z
  \in \CalB(\frac{3}{4} Q_k z_k, \frac{c(k)}{16} \|A\|^{-1})
  \subset \CalB(\frac{3}{4} Q_k z_k, \frac{1}{16} \|A\|^{-1})
\)
for the moment.
Then $\phi(z-x) \neq 0$ implies that
\[
  x
  = -(z - x) + z
  \in \mathcal{B}(0, \tfrac{3}{16} \|A\|^{-1}) + \mathcal{B}(\tfrac{3}{4} Q_k z_k, \tfrac{1}{16} \|A\|^{-1})
  \subseteq \mathcal{B}(\tfrac{3}{4} Q_k z_k, \tfrac{1}{4} \|A\|^{-1})
  ,
\]
so that \Cref{eq:COfKDefinition} implies
\[
  | x |
  \geq \frac{3}{4} | Q_k z_k | - \frac{1}{4} \| A \|^{-1}
  \geq \frac{3}{4} c(k) - \frac{1}{4} c(k)
  = \frac{c(k)}{2}
  ,
\]
and hence $x \in \R^d \setminus \mathcal{B}(0, \frac{c(k)}{2})$.
Using \Cref{eq:valuefk}, it follows therefore that
\begin{align*}
   M^{0, \loc}_{\phi,B} f_k (z)
   & \geq \delta_k
          \int_{\R^d}
            \mathds{1}_{Q_k \mathcal{B}(3/4 z_k , 1/4)} (x)
            \phi(z-x)
          \; dx \\
   & \geq \delta_k
           \Measure \bigl(
                           \mathcal{B}(z, \tfrac{1}{8} \|A\|^{-1})
                           \cap Q_k \mathcal{B}(\tfrac{3}{4} z_k, \tfrac{1}{4})
                         \bigr).
\end{align*}
Now, an application of \cite[Chapter 1, Lemma 10.6]{bownik2003anisotropic}
(with $r = \frac{1}{2} \| A \|^{-1} \leq 1/2$ and $P = \frac{1}{4} Q_k$) yields
because of $\| P \|  r = \frac{\frac{1}{4} \| Q_k \|}{2 \| A \|} \leq \frac{1}{8} \| A \|^{-1}$
and because of $z - \tfrac{3}{4} Q_k z_k \in \CalB(0, \| P \|  \frac{r}{2})$ that
\begin{align*}
  \Measure \bigl(
             \mathcal{B}(z, \tfrac{1}{8} \|A\|^{-1})
             \cap Q_k \mathcal{B}(\tfrac{3}{4} z_k, \tfrac{1}{4})
           \bigr)
  & = \Measure \bigl(
                 \mathcal{B}(z - \tfrac{3}{4} Q_k z_k, \tfrac{1}{8} \|A\|^{-1})
                 \cap \tfrac{1}{4} Q_k \mathcal{B}(0, 1)
               \bigr) \\
  & \geq \Measure \bigl(
                    \mathcal{B}(z - \tfrac{3}{4} Q_k z_k, \| P \| \cdot r)
                    \cap P \mathcal{B}(0, 1)
                  \bigr) \\
  & \geq \left(\frac{r}{2}\right)^d \cdot \Measure \bigl(P \CalB (0,1)\bigr) \\
  & = \left(16 \| A \|\right)^{-d} \cdot |\det Q_k| \cdot \Measure \bigl(\CalB (0,1)\bigr)
  ,
\end{align*}
so that
\begin{align*}
  \delta_k
  \cdot \Measure \bigl(
                   \mathcal{B}(z, \tfrac{1}{8} \|A\|^{-1})
                   \cap Q_k \mathcal{B}(\tfrac{3}{4} z_k, \tfrac{1}{4})
                 \bigr)
  & \geq |\det Q_k|
         \cdot \Measure(\mathcal{B}(0, 1))
          \cdot \delta_k
          \cdot (16 \| A\|)^{-d}.
\end{align*}
Since
\begin{align*}
  |\det Q_k| \cdot \delta_k
  & = \delta_0
      \cdot |\det B|^{\lfloor \eps k \rfloor}
      \cdot |\det A|^{-k -d(k)}
      \cdot |\det B|^{- \frac{\lfloor \eps k \rfloor}{p}}
      \cdot |\det A|^{\frac{k + d(k)}{p}} \\
  & \geq \delta_0 \cdot |\det B|^{\eps k (1 - \frac{1}{p})} \cdot |\det A|^{(k + d(k))(\frac{1}{p} - 1)} \\
  & \gtrsim |\det A|^{k (1 - \frac{1}{p})} |\det A|^{(k + d(k))(\frac{1}{p} - 1)} \\
  & = |\det A|^{d(k)(\frac{1}{p} - 1)} ,
\end{align*}
by definition of $\delta_k$ in \Cref{eq:valuefk} and because $\eps = \ln |\det A| / \ln |\det B|$,
a combination of the above inequalities gives
\[
 \delta_k
  \cdot \Measure \bigl(
                   \mathcal{B}(z, \tfrac{1}{8} \|A\|^{-1})
                   \cap Q_k \mathcal{B}(\tfrac{3}{4} z_k, \tfrac{1}{4})
                 \bigr)
    \gtrsim |\det A|^{d(k) (1/p - 1)}.
\]

Recall that $z \in \CalB (\frac{3}{4} Q_k z_k, \frac{c(k)}{16} \| A \|^{-1})$ was arbitrary.
Thus, combining the estimates obtained above and recalling from \Cref{eq:COfKDefinition}
that $c(k) \geq \| A \|^{-1}$ gives
\begin{align*}
  \| f_k \|_{\LHB}^p
 = \int_{\R^d} \big( M_{\phi, B}^{0, \loc} f_k (z) \big)^p \; dz
 & \geq \int_{\mathcal{B}(\frac{3}{4} Q_k z_k, \frac{c(k)}{16} \|A\|^{-1})}
          \big( M_{\phi, B}^{0, \loc} f_k (z) \big)^p
        \; dz \\
 & \gtrsim |\det A|^{d(k)(1-p)},
\end{align*}
which shows that $\| f_k \|_{\LHB} \to \infty$ as $k \to \infty$,
since $d(k) \to \infty$ and $p < 1$, as well as $|\det A| > 1$.
As noted at the beginning of this step, this completes the proof for the case $p < 1$.
\\~\\
\textbf{Step 2.} \emph{(Case $p = 1$)}.
Since $\| A \|^{-1} \leq c(k) = \| Q_k \| \leq 1$ and $|z_k| = 1$,
by passing to a subsequence if necessary, we can assume that $Q_k \to Q$, as well as
$U_k \to U$ and $z_k \to z^\ast$ for a matrix $Q \in \R^{d \times d}$ satisfying
$\| A \|^{-1} \leq \| Q \| \leq 1$, a vector $z^\ast \in \R^d$ satisfying $|z^\ast| = 1$,
and an orthogonal matrix $U \in \R^{d \times d}$.
Note because of $\eps = \ln |\det A| / \ln |\det B|$ and $d(k) \to \infty$ that
\begin{align*}
  |\det Q_k|
  & = |\det B|^{\lfloor \eps k \rfloor} |\det A|^{-k - d(k)}
    \leq |\det B|^{\eps k} |\det A|^{-k - d(k)} \\
  & = |\det A|^{k} |\det A|^{-k - d(k)}
    = |\det A|^{- d(k)}
    \to 0,
\end{align*}
so that $|\det Q| = 0$, meaning that $Q$ is \emph{not} invertible.

Next, for an arbitrary bounded, continuous function $g \in C_b (\R^d)$, we have
\begin{align*}
  \int_{\R^d}
    f_k (x) g(x)
  \, dx
  & = \int_{\R^d}
        |\det Q_k^{-1}|
        \cdot (D_{U_k^{-1}}^1 f_0) (Q_k^{-1} x)
        \cdot g(Q_k Q_k^{-1} x)
      \, dx \\
  & = \int_{\R^d}
        (D_{U_k^{-1}}^1 f_0) (y)
        g(Q_k y)
      \, dy \\
  & = \int_{\R^d}
        |\det U_k^{-1}|
        f_0 (U_k^{-1} y)
        g(Q_k U_k U_k^{-1} y)
      \, dy \\
  & = \int_{\R^d}
        f_0 (z)
        g(Q_k U_k z)
      \, dz \\
  & \to \int_{\R^d}
          f_0 (z)
          g(Q U z)
        \, d z
     =: \int_{\R^d} g(x) \, d \mu (x)
     ,
\end{align*}
for a uniquely determined regular, real-valued (finite) Borel measure $\mu$ on $\R^d$.
The convergence above follows from the dominated convergence theorem, since
$f_0$ and $g$ are bounded, with $f_0$ of compact support, and since
$g(Q_k U_k z) \to g(Q U z)$ by continuity of $g$.
Note that $\supp \mu \subseteq \mathrm{range} (Q U)$, which is a proper subspace of $\R^d$,
since $Q \in \R^{d \times d}$ is not invertible and thus not surjective.
Hence, $\mu$ is mutually singular with respect to the Lebesgue measure.
Note furthermore that the above implies $f_k \to \mu$ in the sense of tempered distributions.

To show that $\mu \neq 0$, choose $0 < c < \frac{1}{4} \| A \|^{-1}$, and note
\[
  | Q U e_1 |
  = \lim_k | Q_k U_k e_1 |
  = \lim_{k} | Q_k z_k |
  = \lim_k \| Q_k \|
  = \| Q \|
  ,
\]
which implies for any $z \in \CalB (0, \frac{1}{2})$ that
\begin{align*}
  | Q U z - \tfrac{3}{4} Q U e_1 |
  & \geq \tfrac{3}{4} | Q U e_1 | - | Q U z |
    \geq \tfrac{3}{4} \| Q \| - \| Q \| \cdot |U z| \\
  & \geq \tfrac{3}{4} \| Q \| - \tfrac{1}{2} \| Q \|
    =    \tfrac{1}{4} \| Q \|
    \geq \tfrac{1}{4} \| A \|^{-1}
    > c
  .
\end{align*}
Choose a nonnegative, continuous function $g \in C(\R^d)$ satisfying
$\supp g \subseteq \CalB (\frac{3}{4} Q U e_1, c)$ and $g(\frac{3}{4} Q U e_1) = 1$.
By what we just showed, we then have $g(Q U z) = 0$ for all $z \in \CalB (0, \frac{1}{2})$.
By the properties of $f_0$ (see \Cref{eq:functionf0}), we then see
\begin{align*}
  \int_{\R^d} g(x) \, d \mu (x)
  & = \int_{\R^d}
        f_0 (z)
        g(Q U z)
      \, d z \\
  & = \int_{\CalB (0, \frac{1}{2})} f_0 (z) g(Q U z) \, dz
      + \delta_0
        \int_{\CalB(0, \frac{1}{4})}
          g(Q U (\tfrac{3}{4} e_1 + z))
        \, dz \\
  & = \delta_0
      \int_{\CalB(0, \frac{1}{4})}
        g(Q U (\tfrac{3}{4} e_1 + z))
      \, dz
    > 0
     ,
\end{align*}
since the domain of integration is open and the integrand is continuous, nonnegative,
and strictly positive at $z = 0$.

We will now show that the tempered distribution $\mu$ satisfies
$\mu \in h^1 (B) \subseteq L^1$, which will yield the desired contradiction.
For this, fix a nonnegative, nonzero Schwartz function $\phi$.
Then an application of Fatou's lemma yields
\begin{align*}
  \| \mu \|_{\LHBo}
  & \asymp \int_{\R^d} M^{0,\loc}_{\phi,B} \mu (x) \; dx
  \leq  \liminf_{k \to \infty} \int_{\R^d} M_{\phi,B}^{0,\loc} f_{k} (x) \; dx
  \asymp \liminf_{k \to \infty} \| f_{k} \|_{\LHBo}.
\end{align*}
Since $\| f_k \|_{\LHB} \lesssim 1$ for all $k \in \mathbb{N}$ by \Cref{lem:test_function_hardy},
this shows that $\mu \in \LHBo \subseteq L^1$, which is a contradiction,
since $\mu \neq 0$ is mutually singular with respect to the Lebesgue measure.
\end{proof}

\begin{remark}\label{rem:example}
  While being based on the same general ideas, the proof for the case $p = 1$ above
  adds a significant detail that was missing in the proof
  of \cite[Chapter 1, Theorem 10.5]{bownik2003anisotropic}.
  The reason is that one of the claims used in \cite{bownik2003anisotropic} appears not correct
  as stated:
  In \cite{bownik2003anisotropic}, it is effectively claimed that if $(f_n)_{n \in \N}$ is a sequence
  in $L^1$ with uniformly bounded supports
  that converges in the sense of tempered distributions to some real-valued Borel measure $\mu$,
  and such that $\Measure (\suppc f_n) \to 0$ as $n \to \infty$, then $\mu$ is
  mutually singular with respect to the Lebesgue measure.

  To see that this claim is not correct in general, let $f_n : \R \to [0, \infty)$, $n \in \N$,
  be defined by
  \[
    f_n (x)
    = \frac{1}{n}
      \sum_{i = 1}^n
        \frac{n^2}{2} \Indicator_{\frac{i}{n} + [-n^{-2}, n^{-2}]}.
  \]
  Then $\| f_n \|_{L^1} = 1$, and $\Lebesgue{\suppc f_n} \leq \frac{2}{n}$,
  so that $\Lebesgue{\suppc f_n} \to 0$ as $n \to \infty$.,
  However, it follows by standard arguments hat $f_n \to \Indicator_{[0,1]}$
  in the weak-$\ast$-topology of $M(\R) = (C_0 (\R))^*$,
  so that $\lim_n f_n \in L^1$ is \emph{not} singular with respect to the Lebesgue measure $\Measure$.
\end{remark}

\subsection{Proof of Theorem \ref{thm:necessary}}

Combining the results from the previous subsections, we can prove \Cref{thm:necessary}.

\begin{proof}[Proof of \Cref{thm:necessary}]
If $\TL(A) = \TL(B)$ for some $\alpha \neq 0$, then case (i) follows by \Cref{thm:casenonzero}.
If $\TLzero(A) = \TLzero(B)$ for some $p \in (0, \infty)$ and $q \neq 2$,
then case (i) follows from \Cref{thm:casenzero1}.
Lastly, if $\TLzero(A) = \TLzero(B)$ for $p \in (0, 1]$ and $q = 2$,
then case (i) follows from \Cref{thm:casenzero2}, combined with \Cref{prop:TL_hardy}.
In the remaining case, we have $\alpha = 0$, $q = 2$, and $p \in (1,\infty)$,
so that case (ii) of \Cref{thm:necessary} holds.
\end{proof}

\subsection{Proof of Theorem \ref{thm:main}}
\label{sub:MainTheoremProof}
 \Cref{thm:necessary} shows that (i) implies (iii), whereas \Cref{prop:sufficient} shows that (iii) implies (ii). The remaining implication is immediate.

\appendix

\section{Postponed proofs}

\begin{lemma}\label{lem:SpecialFunctionExists}
  Let $s,d \in \N$, and let $e_1 = (1,0,\dots,0) \in \R^d$ denote the first standard basis vector.
  There exists a bounded measurable function $f : \R^d \to \CC$ satisfying
  \begin{align*}
  f (x) =
  \begin{cases}
    1, & \text{if } x \in \mathcal{B}(\frac{3}{4} e_1, \frac{1}{4}) \\
    0,        & \text{if } x \notin \mathcal{B}(0, \frac{1}{2}) \cup \mathcal{B}(\frac{3}{4} e_1, \frac{1}{4})
  \end{cases}
\end{align*}
and $\int_{\R^d} f (x) x^{\sigma} \; dx = 0$ for all $\sigma \in \N_0^d$ with $|\sigma| \leq s$.
\end{lemma}

\begin{proof}
Define $n_{s} := |\{ \sigma \in \N_0^d : |\sigma| \leq s \}|$ and $v \in \R^{n_s}$
by $v_{\sigma} := - \int_{\mathcal{B}(\frac{3}{4} e_1, \frac{1}{4})} x^{\sigma} \; dx$.
Then, the linear function
\[
  \theta : \quad
  L^{\infty} (\mathcal{B}(0,1/2)) \to \R^{n_s}, \quad
  h \mapsto \bigg( \int_{\mathcal{B}(0,\frac{1}{2})} h(x) x^{\sigma} \; dx \bigg)_{|\sigma| \leq s}
\]
is surjective.
Indeed, if this was not the case, since $\mathrm{range}(\theta)$ is finite-dimensional,
there would exist $c \in \R^{n_s}$ with $c \neq 0$ but $c \perp \mathrm{range}(\theta)$,
which then implies for the nonzero polynomial $p(x) := \sum_{|\sigma| \leq s} c_\sigma \, x^\sigma$
that $0 = \int_{\CalB (0, \frac{1}{2})} h(x) p(x) \, d x$ for all $h \in L^\infty (\CalB(0, \frac{1}{2}))$,
which is absurd.

Hence, there exists $h \in L^{\infty}(\mathcal{B}(0, 1/2))$ such that $\theta(h) = v$.
Define $f : \R^d \to \mathbb{C}$ by
\[
 f (x) =
 \begin{cases}
  h(x) & \text{if }  x \in \mathcal{B}(0, \frac{1}{2}), \\
  1    & \text{if }  x \in \mathcal{B}(\frac{3}{4} e_1, \frac{1}{4}), \\
  0    & \text{if } x \notin \mathcal{B}(0, \frac{1}{2}) \cup \mathcal{B}(\frac{3}{4} e_1, \frac{1}{4}).
 \end{cases}
\]
Then, given $\sigma \in \N_0^d$ with $|\sigma| \leq s$, we have
\[
  \int_{\R^d} f(x) x^{\sigma} \; dx
  = \int_{\mathcal{B}(0, \frac{1}{2})} h(x) x^{\sigma} \; dx
    + \int_{\mathcal{B}(\frac{3}{4} e_1, \frac{1}{4})} x^{\sigma} \; dx
  = 0,
\]
as desired.
\end{proof}

The next lemma is part of the folklore. However, since we could not locate a reference
and since the properties derived in the lemma are crucial for our arguments
(see \Cref{lem:NormEquivalence}), we provide a short proof.

\begin{lemma}\label{lem:Completeness}
  Let $A \in \GL (d, \R)$ be expansive and let $\alpha \in \R$, $p \in (0,\infty)$,
  and $q \in (0,\infty]$.
  Then the following assertions hold:
  \begin{enumerate}
    \item[(i)] The quasi-norm $\| \cdot \|_{\TLA}$ is an $r$-norm for $r := \min \{ 1, p, q \}$,
               that is,
              \[
                \| f_1 + f_2 \|_{\TLA}^r \leq \| f_1 \|_{\TLA}^r + \| f_2 \|_{\TLA}^r
              \]
              for all $f_1, f_2 \in \TLA$;

    \item[(ii)] The space $\TLA$ is complete with respect to the quasi-norm $\| \cdot \|_{\TLA}$;

    \item[(iii)] If $(f_n)_{n \in \N}$ is a sequence in $\TLA$ satisfying $f_n \to f_0$ with convergence in $\TLA$,
                 then $\langle f_n, \phi \rangle \!\to\! \langle f_0, \phi \rangle$
                 for all $\phi \in \Fourier (C_c^\infty (\R^d))$.
  \end{enumerate}
\end{lemma}

\begin{proof}
(i) Let $r := \min \{ 1, p, q \}$.
To ease notation, define $C_f (x) := (|\det A|^{\alpha i} |f \ast \varphi_i^A (x)|)_{i \in \N_0}$
for $f \in \mathcal{S}'$ and $x \in \R^d$.
Then
\begin{align*}
 \| f \|_{\TL}^r
 = \big\| \| C_f (\cdot)  \|_{\ell^q} \big\|_{L^p}^r
 = \big\| \big\| \big( C_f (\cdot) \big)^r \big\|_{\ell^{q/r}}^{1/r} \big\|_{L^p}^r
 = \big\| \big\| \big( C_f (\cdot) \big)^r \big\|_{\ell^{q/r}} \big\|_{L^{p/r}}.
\end{align*}
Since $r \leq 1$, we have $C_{f_1 + f_2}^r \leq C_{f_1}^r + C_{f_2}^r$.
Since moreover $q/r, p/r \geq 1$, applications of the triangle inequality yield
\[
   \| f_1 + f_2 \|_{\TL}^r
   \leq \big\| \big\| \big( C_{f_1} (\cdot) \big)^r \big\|_{\ell^{q/r}} \big\|_{L^{p/r}}
        + \big\| \big\| \big( C_{f_2} (\cdot) \big)^r \big\|_{\ell^{q/r}} \big\|_{L^{p/r}}
   =    \| f_1 \|_{\TL}^r + \| f_2 \|_{\TL}^r,
\]
as required.
\\~\\
(ii) Let $\mathcal{D} := \{ D = A^j ([0,1]^d + k) : j \in \Z, k \in \Z^d \}$
and $\mathcal{D}_0 := \{ D \in \mathcal{D} : \Lebesgue{D} \leq 1 \}$.
For a complex-valued sequence $c = (c_D)_{D \in \mathcal{D}_0}$, define
\[
  \| c \|_{\TLseqA}
  := \bigg\|
       \bigg(
         \sum_{D \in \mathcal{D}_0}
         \big( \Lebesgue{D}^{-\alpha - 1/2} |c_D | \Indicator_D \big)^q
       \bigg)^{1/q}
     \bigg\|_{L^p}
  \in [0,\infty],
\]
and set $\TLseqA := \{ c \in \mathbb{C}^{\mathcal{D}_0} : \| c \|_{\TLseqA} < \infty \}$.
Then it is easily verified that $\TLseqA$, with respect to the quasi-norm $\| \cdot \|_{\TLseqA}$,
is a (solid) quasi-normed function space on $\mathcal{D}_0$,
in the sense of \cite[Section 2.2]{VoigtlaenderPhDThesis} and \cite[Section 2]{lorist2023banach}.
Moreover, $\TLseqA$ satisfies the Fatou property, and hence it is complete,
see, e.g., \cite[Lemma 2.2.15]{VoigtlaenderPhDThesis} and \cite[Proposition 2.2]{lorist2023banach}
(combined with \cite[Remark 2.1]{lorist2023banach}).

By \cite[Section 3.3]{bownik2006atomic}, there exist two bounded linear maps
$S : \TLA \to \TLseqA$ and $T : \TLseqA \to \TLA$ satisfying $T \circ S = \mathrm{id}_{\TLA}$.
Hence, if $(f_n)_{n \in \N}$ is a Cauchy sequence in $\TLA$,
then the sequence $(c^{(n)})_{n \in \N}$ given by $c^{(n)} = S f_n \in \TLseqA$ is Cauchy in $\TLseqA$,
and thus  converges to some $c \in \TLseqA$. Since $T$ is bounded, this easily implies that
$f_n = T (Sf_n) = T (c^{(n)}) \to Tc \in \TLA$, which shows that $\TLA$ is complete.
\\~\\
(iii)
Choose an $A$-analyzing pair $(\varphi, \Phi)$  that satisfies properties
\ref{enu:AnalyzingPairSupportCondition}-\ref{enu:AnalyzingPairPartitionOfUnity}.
Let $(f_n)_{n \in \N}$ be a sequence in $\TLA$ that converges in $\TLA$ to some $f_0 \in \TLA$.
Let $\phi \in \mathcal{F}(C_c^{\infty} (\R^d))$ and note by elementary properties
of the Fourier transform (see \cite[Theorem 7.19]{RudinFA})
and because of $\sum_{i \in \N_0} \widehat{\varphi_i^A} \equiv 1$
(see property \ref{enu:AnalyzingPairPartitionOfUnity}) that
\begin{align*}
  \langle f_n, \phi \rangle
  = \langle \widehat{f_n}, \mathcal{F}^{-1} \phi \rangle
  = \sum_{i \in \N_0} \langle \widehat{f_n}, \widehat{\varphi_i^A} \cdot \mathcal{F}^{-1} \phi \rangle
  = \sum_{i \in \N_0} \langle f_n \ast \varphi_i^A, \phi \rangle
\end{align*}
for any $n \in \N_0$.
Moreover, there exists a finite set $I_\phi \subseteq \N_0$ independent of $n$,
such that $\langle f_n \ast \varphi_i^A, \phi \rangle = 0$ for all $n \in \N_0$
and all $i \in \N_0 \setminus I_\phi$.
Thus, it remains to show 
\[
  \langle f_n \ast \varphi_i^A , \phi \rangle
  \to \langle f_0 \ast \varphi_i^A, \phi \rangle
\]
as $n \to \infty$, for all $i \in \N_0$.
To see this, we will use \cite[Corollary 3.2]{bownik2005atomic},
which yields a constant $C>0$ and some $N \in \N$ such that
\[
  \sup_{x \in \R^d} \frac{|h(x)|}{(1+|x|)^N}
  \leq C^{i + 1} \| h \|_{L^p}
  \quad \text{
          for all $h \in \mathcal{S}'(\R^d)$
          with $\suppc \widehat{h} \subseteq \suppc \widehat{\varphi_i^A}$.
        }
\]
Hence, in particular,
\begin{align*}
  \sup_{x \in \R^d} \frac{|(f_0 \ast \varphi_i^A - f_n \ast \varphi_i^A)(x)|}{(1+|x|)^N}
  &\leq C^{i+1} \| f_0 \ast \varphi_i^A - f_n \ast \varphi_i^A \|_{L^p} \\
  &\leq C^{i+1} |\det A|^{- \alpha i} \| f_0 - f_n \|_{\TLA},
\end{align*}
which easily implies that $f_n \ast \varphi_i^A \to f_0 \ast \varphi_i^A$
in $\mathcal{S}' (\R^d)$, as $n \to \infty$.
Thus, we see that
$\langle f_n \ast \varphi_i^A, \phi \rangle \to \langle f_0 \ast \varphi_i^A, \phi \rangle$,
as desired.
\end{proof}

\section*{Acknowledgments}
For J.v.V., this research was funded in whole or in part by the Austrian Science Fund (FWF): 10.55776/J4555.
J.v.V. is grateful for the hospitality and support
of the Katholische Universität Eichstätt-Ingolstadt during his visit.
F. V. acknowledges support by the Hightech Agenda Bavaria.

\end{document}